\documentclass[11pt]{amsart}
\pdfoutput=1
\usepackage{amssymb,amscd,graphicx,enumerate,epsfig}
\usepackage[bookmarks=true]{hyperref}

\newcommand{\Claim}[1]{\noindent\textbf{Claim. }\textit{#1}\\}
\newtheorem{theorem}{Theorem}[section]

\newtheorem{proposition}[theorem]{Proposition}

\newtheorem{lemma}[theorem]{Lemma}

\theoremstyle{definition}
\newtheorem{definition}[theorem]{Definition}
\newenvironment{proofc}{\noindent\textit{Proof of Claim.}}{\\}

\newcommand{\Case}[1]{\textbf{Case #1.}}
\begin{document}

\title[Finite group actions on Haken Manifolds]{Handlebody-preserving finite group actions on Haken manifolds with Heegaard genus two.}
\author{Jungsoo Kim}
\date{}
\begin{abstract} Let $M$ be a closed orientable $3$-manifold with a genus two Heegaard splitting $(V_1, V_2; F)$ and a non-trivial JSJ-decomposition, where all components of the intersection of the JSJ-tori and $V_i$ are not $\partial$-parallel in $V_i$ for $i=1,2$. If $G$ is a finite group of orientation-preserving diffeomorphisms acting on $M$ which preserves each handlebody of the Heegaard splitting and each piece of the JSJ-decomposition of $M$, then $G\cong \mathbb{Z}_2$ or $\mathbb{D}_2$ if $V_j\cap(\cup T_i)$ consists of at most two disks or at most two annuli.
\end{abstract}

\address{\parbox{4in}{Department of Mathematics\\
Konkuk University\\
Seoul 143-701\\ Korea\medskip}} \email{pibonazi@gmail.com}

\maketitle
\tableofcontents

\section[Introduction]{Introduction}
Let $M$ be a closed orientable $3$-manifold of Heegaard genus two with a non-trivial JSJ-decomposition and $G$ be a finite group of orientation-preserving diffeomorphisms acting on $M$ which preserves each handlebody of the Heegaard splitting. Let $(V_1,V_2;F)$ be the Heegaard splitting and $\cup T_i$ be the union of the JSJ-tori.

Now we get a question that how we can classify such $G$ under proper criteria.
If $G$ preserves each piece of the JSJ-decomposition, then the quotient orbifold $M/G$ can be assumed as the union of quotient orbifolds, where each quotient orbifold is induced from the quotient of each piece of the JSJ-decomposition by $G$. In this context, we will give a restriction to $G$, i.e. $G$ preserves each piece of the JSJ-decomposition.

D. McCullough, A. Miller and B. Zimmermann found a necessary and sufficient condition under which a finite group of orientation-preserving diffeomorphisms acts on a handlebody by the concepts \textit{``$G$-admissible graphs of groups''} and \textit{``the Euler characteristic of a graph of groups''} in \cite{MMZ}. In particular, J. Kalliongis and A. Miller characterized the finite groups of diffeomorphisms acting on a solid torus in \cite{KM}. Also, there is an exact characterization for the quotient orbifolds from a genus two handlebody by finite groups of orientation-preserving diffeomorphisms using the idea of \cite{MMZ} in \cite{KJS}.

In this article, we will determine the possible isomorphism types of finite groups of orientation-preserving diffeomorphisms acting on $M$ which preserve each $V_i$ for $i=1,2$ and each piece of the JSJ-decomposition by the following theorem.

\begin{theorem}[the Main theorem]\label{main-theorem}
Let $M$ be a closed orientable $3$-manifold with a genus two Heegaard splitting $(V_1, V_2; F)$ and a non-trivial JSJ-decomposition, where all components of the intersection of the JSJ-tori and $V_i$ are not $\partial$-parallel in $V_i$ for $i=1,2$. If $G$ is a finite group of orientation-preserving diffeomorphisms acting on $M$ which preserves each handlebody of the Heegaard splitting and each piece of the JSJ-decomposition of $M$, then $G\cong \mathbb{Z}_2$ or $\mathbb{D}_2$ if $V_j\cap(\cup T_i)$ consists of at most two disks or at most two annuli.
\end{theorem}

In Theorem of p.437 of \cite{KO}, T. Kobayashi characterized the structure of closed, orientable $3$-manifolds with Heegaard genus two which have non-trivial JSJ-decompositions. By this theorem, the possible structures of $M$ can be described as follows.

\begin{enumerate}
    \item  $M=M_1\cup_T M_2$, $T$ is a separating JSJ-torus in $M$. ((i), (ii) and (iii) of Theorem of \cite{KO}.)
    \item $M=M_1\cup_{T_1}M_2\cup_{T_2}M_3$ where each $T_i$ is a separating JSJ-torus in $M$ for $i=1,2$. ((iv) of Theorem of \cite{KO}.)
    \item $M=M_1\cup_{T_1\, ,T_2} M_2$ where each $T_i$ is a non-separating JSJ-torus in $M$ for $i=1,2$. ((v) of Theorem of \cite{KO}.)\\
    In all cases, each $M_i$ is a piece of JSJ-decomposition for $i=1,2,3$.
\end{enumerate}

In particular, T. Kobayashi used a sequence of isotopies of type A which move $\cup T_i$ so that $V_1\cap(\cup T_i)$ consists of minimal number of disks (see p.24 of \cite{JA} for the term ``\textit{isotopy of type A}''), and divided the proof of his theorem into some cases using this idea. In particular, he proved that $V_1\cap(\cup T_i)$ consists of at most two disks if the number of disks is minimal when the JSJ-tori are separating.
Moreover, he also used an additional sequence of isotopies of type A so that $V_1\cap(\cup T_i)$ consists of minimal number of annuli. So we get two conditions for the JSJ-tori and the Heegaard splitting as follows.
\begin{enumerate}[(A).]
    \item $V_1\cap(\cup T_i)$ consists of disks, and the number of disks is at most two,
    \item $V_1\cap(\cup T_i)$ consists of annuli, and the number of annuli is at most two,\\
where we assume that each component of $V_1\cap(\cup T_i)$ is not $\partial$-parallel in $V_1$.
\end{enumerate}
Let the first be \emph{``Condition A''} and the second be \emph{``Condition B''}. We will prove Theorem \ref{main-theorem} if the Heegaard splitting and the JSJ-tori satisfy Condition A or Condition B.

\begin{center}
Acknowledgement
\end{center}
This work was supported by the Korea Science and Engineering
Foundation (KOSEF) grant funded by the Korea government (MOST)
(No.~R01-2007-000-20293-0).

\section{Some geometric lemmas\label{section-gms}}

We first introduce the term \textit{handlebody orbifold}.\\

\begin{definition}(\cite{Z1}, pp.594)\label{definition-first}
A \emph{handlebody orbifold} consists of finitely many orbifolds as in Figure \ref{fig-z1} (i.e., quotients of finite orthogonal group actions on the $3$-ball) connected by 1-handle orbifolds respecting the singular axes and their orders, and such that topologically the result is an orientable handlebody.
\end{definition}

\begin{figure}
\includegraphics[viewport=28 671 550 781, width=12cm]{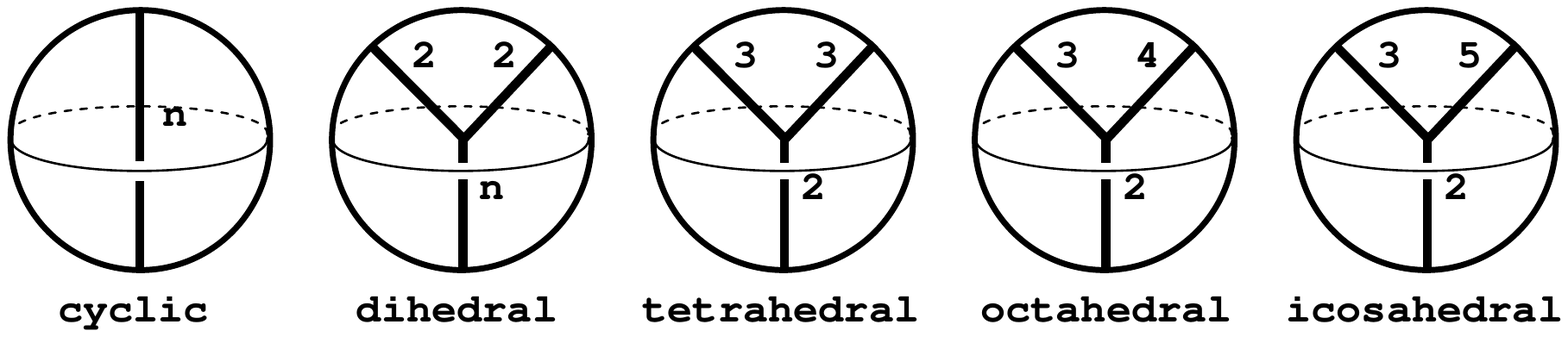}
\caption{The quotients of $3$-ball by spherical groups\label{fig-z1}}
\end{figure}

We can interpret the quotient of a handlebody by a finite group action as a handlebody orbifold by the following proposition.

\begin{proposition}(\cite{Z1}, Proposition 1)
The quotients of handlebodies by finite group actions are exactly the handlebody orbifolds. \label{proposition-Z1}
\end{proposition}

D. McCullough, A. Miller and  B. Zimmermann proved the following Proposition.

\begin{proposition}[\cite{MMZ}, Theorem 8.2 (b)] Up to isomorphism the finite groups which act on a handlebody of genus two are the subgroups of $\mathbb{D}_4$ and $\mathbb{D}_6$.\label{proposition-MMZ}
\end{proposition}

In \cite{KJS}, we can classify all handlebody orbifolds from a genus two handlebody by orientation-preserving finite group actions.

\begin{proposition}(\cite{KJS}, Corollary 2.10\label{corollary-KJS})
Let $G$ be a finite group of orientation-preserving diffeomorphisms which acts on a genus two handlebody. Then all possible handlebody orbifolds from a genus two handlebody by $G$ are listed in Figure \ref{fig-ho}.
\end{proposition}

\begin{figure}
\includegraphics[viewport=23 128 480 778, width=8cm]{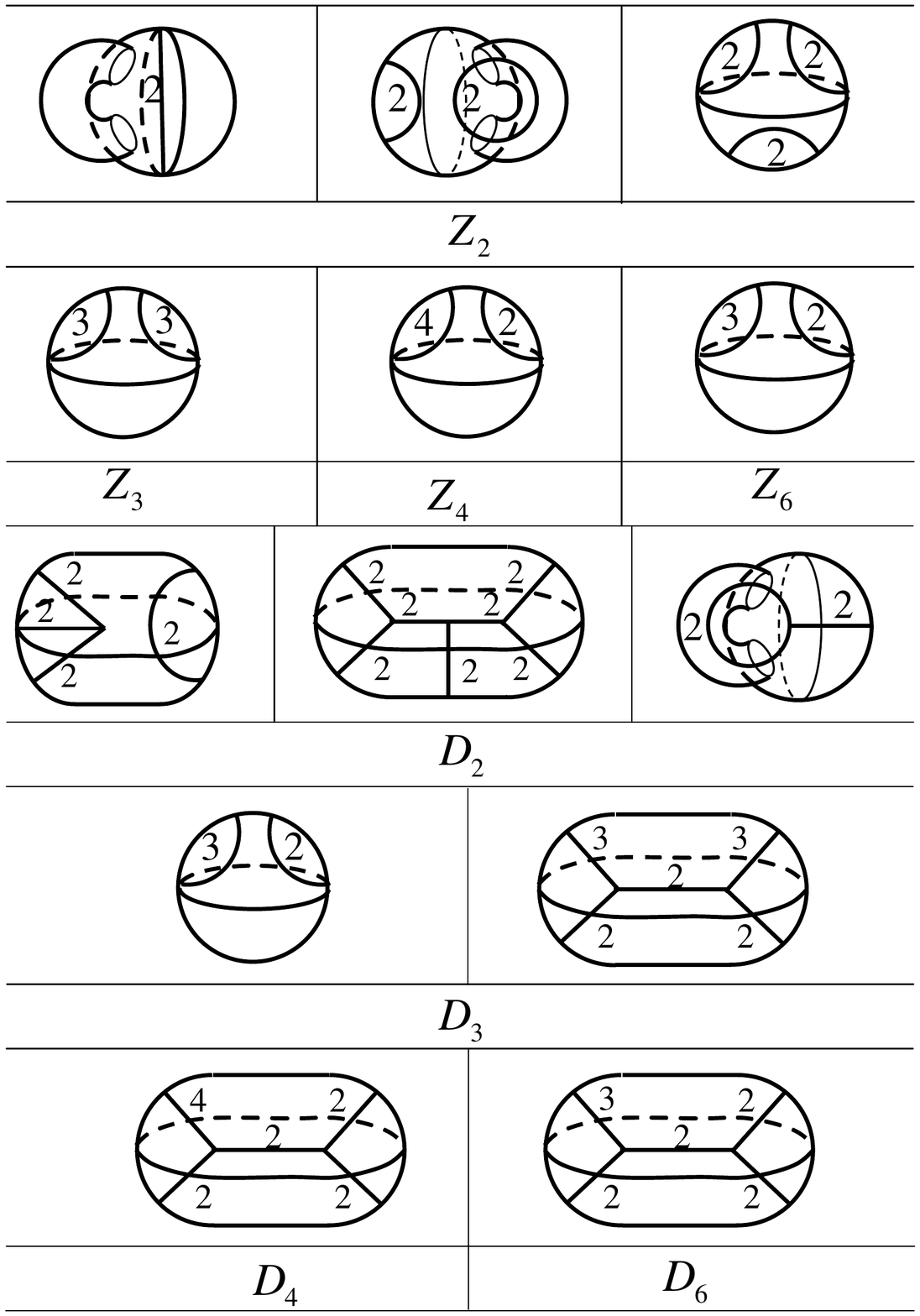}
\caption{All handlebody orbifolds from a genus two handlebody.\label{fig-ho}}
\end{figure}

J. Kalliongis and A. Miller derived all possible quotient orbifolds from a solid torus by orientable preserving finite group actions by the following lemma.

\begin{lemma}\label{lemma-KM}(\cite{KM}, p.377 or Table 2 of p.394)
Let $G$ be a group of orientation preserving finite smooth actions on a solid torus. Then all possible quotient orbifolds by $G$ are classified into the two types $V(A0,k)$ and $V(B0,k)$ as Figure \ref{figure-KM}.
\end{lemma}

\begin{figure}
\centering
\includegraphics[viewport=24 574 567 781,width=10cm]{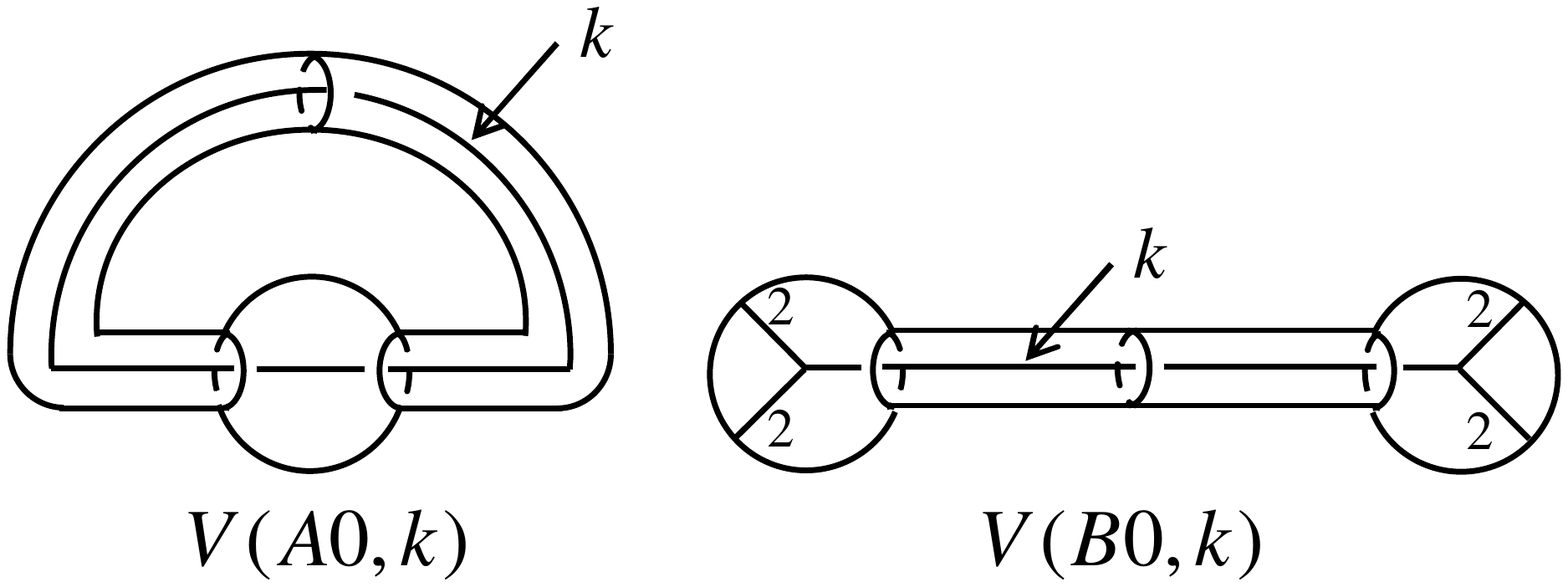}\\
\caption{All handlebody orbifolds from a solid
torus.\label{figure-KM}}
\end{figure}

From now on, we will assume that $G$ acts on a genus two handlebody, so $G$ is a subgroup of $\mathbb{D}_4$ or $\mathbb{D}_6$, i.e. one of $\mathbb{Z}_2$, $\mathbb{Z}_3$, $\mathbb{Z}_4$, $\mathbb{Z}_6$, $\mathbb{D}_2$, $\mathbb{D}_3$, $\mathbb{D}_4$ and $\mathbb{D}_6$ by Proposition \ref{proposition-MMZ}.

If $G$ acts on a solid torus, then $G\cong \mathbb{Z}_m\times\mathbb{Z}_l$ or $\operatorname{dih}(\mathbb{Z}_m\times\mathbb{Z}_l)$ in the first paragraph of p.378 of \cite{KM}.

Suppose that a solid torus $\bar{V}$ is embedded in a genus two handlebody $V$ and $G$ acts on both $V$ and $\bar{V}$.

If we use Lemma 2.1 of \cite{KM}, Lemma 2.5 of \cite{KM} and the orientable parts of Table 2 in p.394 of \cite{KM}, then we get the left side of Table \ref{table-sub} if $G$ is abelian (so $G\cong\mathbb{Z}_2$, $\mathbb{Z}_3$, $\mathbb{Z}_4$, $\mathbb{Z}_6$ or $\mathbb{D}_2$) and the right side of Table \ref{table-sub} if $G$ is non-abelian ($G\cong\mathbb{D}_3$, $\mathbb{D}_4$ or $\mathbb{D}_6$.) Denote that $\mathbb{D}_2\cong \mathbb{Z}_2\times\mathbb{Z}_2$ and $\mathbb{Z}_2\cong \mathbb{D}_1$. See section 2 of \cite{KM} for more details.

\begin{table}
\begin{tabular}{ll}
$\bar{V}/G$ & $G$\\\hline
$V(A0,1)$: & $\mathbb{Z}_2$, $\mathbb{Z}_3$, $\mathbb{Z}_4$ and $\mathbb{Z}_6$. \\
$V(A0,2)$: & $\mathbb{Z}_2$, $\mathbb{Z}_4$ and $\mathbb{Z}_6$.\\
$V(A0,3)$: & $\mathbb{Z}_3$ and $\mathbb{Z}_6$.\\
$V(A0,4)$: & $\mathbb{Z}_4$.\\
$V(A0,6)$: & $\mathbb{Z}_6$.\\
$V(B0,1)$: & $\mathbb{Z}_2$ and $\mathbb{D}_2$.\\
$V(B0,2)$: & $\mathbb{D}_2$.
\end{tabular}
\begin{tabular}{ll}
$\bar{V}/G$ & $G$\\\hline
$V(B0,1)$: & $\mathbb{D}_3$, $\mathbb{D}_4$ and $\mathbb{D}_6$.\\
$V(B0,2)$: & $\mathbb{D}_4$ and $\mathbb{D}_6$.\\
$V(B0,3)$: & $\mathbb{D}_3$ and $\mathbb{D}_6$.\\
$V(B0,4)$: & $\mathbb{D}_4$.\\
$V(B0,6)$: & $\mathbb{D}_6$\\
&\\
&
\end{tabular}
\caption{$\bar{V}/G$ for abelian (left) and non-abelian cases (right).\label{table-sub}}
\end{table}

So we get Table \ref{table-1} by rearranging Table \ref{table-sub}, where Table \ref{table-1} exhibits possible quotient types of $\bar{V}/G$ for each isomorphism type of $G$.

\begin{table}
\begin{tabular}{ll}
$\mathbb{Z}_2$ : & $V(A0,1)$, $V(A0,2)$, $V(B0,1)$. \\
$\mathbb{Z}_3$ : & $V(A0,1)$, $V(A0,3)$. \\
$\mathbb{Z}_4$ : & $V(A0,1)$, $V(A0,2)$, $V(A0,4)$. \\
$\mathbb{Z}_6$ : & $V(A0,1)$, $V(A0,2)$, $V(A0,3)$, $V(A0,6)$.\\
$\mathbb{D}_2$ : & $V(B0,1)$, $V(B0,2)$.\\
$\mathbb{D}_3$ : & $V(B0,1)$, $V(B0,3)$.\\
$\mathbb{D}_4$ : & $V(B0,1)$, $V(B0,2)$, $V(B0,4)$.\\
$\mathbb{D}_6 $: & $V(B0,1)$, $V(B0,2)$, $V(B0,3)$, $V(B0,6)$.
\end{tabular}
\caption{Possible quotient types of $\bar{V}/G$. \label{table-1}}
\end{table}

From now, we will introduce some geometric lemmas which will be used in the proof of Theorem \ref{main-theorem}.

\begin{lemma}\label{lemma-gm-first}
Let $G$ be a finite cyclic group of non-trivial orientation preserving diffeomorphisms which acts on a genus two handlebody $V$. Suppose that $\mathcal{S}$ is a set of properly embedded surfaces in $V$ where $\mathcal{S}$ cuts $V$ into one or more components, where there is a solid torus component $V_1$ among them. If $G$ preserves $V_1$, then $\overline{V_1}/G$ is $V(A0,1)$ unless $G\cong\mathbb{Z}_2$.
\end{lemma}

\begin{proof}
Since $G$ preserves $V_1$, $V/G$ can be written as
$$V/G=\overline{V_1}/G \cup [V-V_1]/G,$$
where $\overline{V_1}/G\cap [V-V_1]/G$ appears only on the boundary of each quotient.

In the case of $G\cong \mathbb{Z}_3$, $\mathbb{Z}_4$ or $\mathbb{Z}_6$, $V_1/G$ is $V(A0,k)$ for $k\geq 1$ by Table \ref{table-1}. But there is no isolated circle in the singular locus of $V/G$ unless $G\cong\mathbb{Z}_2$ by Proposition \ref{corollary-KJS}. So if $G$ is not isomorphic to $\mathbb{Z}_2$, then $\overline{V_1}/G$ must be $V(A0,1)$.
\end{proof}

\begin{lemma}\label{lemma-gm-1}
Let  $G$ be a finite group of non-trivial orientation preserving diffeomorphisms which acts on a genus two handlebody $V$. Suppose that $\mathcal{S}$ is a set of properly embedded surfaces in $V$ where $\mathcal{S}$ cuts $V$ into two solid tori $V_1$ and $V_2$. If $G$ preserves both $V_1$ and $V_2$, then $G\cong  \mathbb{Z}_2$.
\end{lemma}

\begin{proof}
Since $G$ preserves both $V_1$ and $V_2$, $V/G$ can be written as
$$V/G=\overline{V_1}/G \cup \overline{V_2}/G,$$
where $\overline{V_1}/G\cap \overline{V_2}/G$ appears only on the boundary of each quotient.
In the case of $G\cong \mathbb{Z}_3$, $\mathbb{Z}_4$ or $\mathbb{Z}_6$, $\overline{V_i}/G$ is $V(A0,1)$ by Lemma \ref{lemma-gm-first}. So the singular locus of $V/G$ consists of at most two isolated circles, but this contradicts Proposition \ref{corollary-KJS}.

In the case of $G\cong \mathbb{D}_2$, the singular locus of each $\overline{V_i}/G$ has even vertices of valency $3$ by Table \ref{table-1}, so the singular locus of $V/G$ also has even vertices of valency $3$. But this contradicts Proposition \ref{corollary-KJS}. In the case of $\mathbb{D}_3$, $\mathbb{D}_4$ or $\mathbb{D}_6$, there must be an index 2 cyclic subgroup $\bar{G}$ of $G$. But $|\bar{G}|\geq3$ and $\bar{G}$ also preserves $V_i$ for $i=1,2$, we get a contradiction using the same arguments in the case of $G\cong \mathbb{Z}_3$, $\mathbb{Z}_4$ or $\mathbb{Z}_6$.
\end{proof}

\begin{lemma}\label{lemma-gm-12}
Let  $G$ be a finite group of non-trivial orientation preserving
diffeomorphisms which acts on a genus two handlebody $V$. Suppose that $\mathcal{S}$ is a set of properly embedded surfaces in $V$ where $\mathcal{S}$ cuts $V$ into two solid tori $V_1$, $V_3$ and a $3$-ball $V_2$. If $G$ preserves each $V_i$ for $i=1,2,3$, or $G$ preserves $V_2$ but exchanges $V_1$ and $V_3$, then $G\cong\mathbb{Z}_2$ or $\mathbb{D}_2$.
\end{lemma}

\begin{proof}
Suppose that $G$ preserves $V_i$ for $i=1,2,3$. Then $V/G$ can be written as
$$V/G=\overline{V_1}/G \cup \overline{V_2}/G \cup \overline{V_3}/G,$$
where $\overline{V_i}/G\cap \overline{V_j}/G$ appears only on the boundary of each quotient ($i\neq j$ and $1\leq i,j\leq 3$.)

In the case of $G\cong \mathbb{Z}_3$, $\mathbb{Z}_4$ or $\mathbb{Z}_6$,
the singular locus of each $\overline{V_i}/G$ is $V(A0,1)$ for $i=1,3$ by Lemma \ref{lemma-gm-first}, and the singular locus of $\overline{V_2}/G$ is an arc of index larger than $2$ by Figure \ref{fig-z1}. So the singular locus of $V/G$ consists of only one arc, this contradicts Proposition \ref{corollary-KJS}.

In the case of $G\cong \mathbb{D}_3$, $\mathbb{D}_4$ or $\mathbb{D}_6$, there must be an index 2 cyclic subgroup $\bar{G}$ of $G$. But $|\bar{G}|\geq3$ and $\bar{G}$ also preserves $V_i$ for $i=1,2$, we get a contradiction using the same arguments in the case of $G\cong \mathbb{Z}_3$, $\mathbb{Z}_4$ or $\mathbb{Z}_6$.

Suppose that $G$ preserves $V_2$ but exchanges $V_1$ and $V_3$. Then $V/G$ can be written as
$$V/G=\overline{V_1}/G' \cup \overline{V_2}/G,$$
where $G'$ is an index two subgroup of $G$ which preserves both $V_1$ and $V_3$, and $\overline{V_1}/G'\cap \overline{V_2}/G$ appears only on the boundary of each quotient.

In the case of $G\cong \mathbb{Z}_3$, the index two subgroup $G'$ cannot exist. So we get a contradiction. In the case of $G\cong \mathbb{Z}_4$, we get $G'\cong \mathbb{Z}_2$. So $V_2/G$ is a $3$-ball with an arc of index $4$ by  Figure \ref{fig-z1}, and $\overline{V_1}/G'$ is $V(A0,1)$, $V(A0,2)$ or $V(B0,1)$ by Table \ref{table-1}. The first case contradicts Proposition \ref{corollary-KJS} since the singular locus of $V/G$ consists of an arc. The second case contradicts Proposition \ref{corollary-KJS} from the existence of an isolated circle in the singular locus of $V/G$. The third case also contradicts Proposition \ref{corollary-KJS} from the existence of three arcs in the singular locus in $V/G$ (the arc of index $4$ in $\overline{V_2}/G$ cannot be extended to $\overline{V_1}/G'\cap \overline{V_2}/G$ from the nonexistence of a point of index $4$ in the boundary of $\overline{V_1}/G'$.) In the case of $G\cong \mathbb{Z}_6$, the singular locus of $\overline{V_2}/G$ is an arc of index $6$ by Figure \ref{fig-z1}, i.e. the singular locus of $V/G$ has an edge of index $6$. But this contradicts Proposition \ref{corollary-KJS}.

In the case of $G\cong \mathbb{D}_3$, we get $G'\cong \mathbb{Z}_3$. So $\overline{V_1}/G'$ is $V(A0,1)$ by Lemma \ref{lemma-gm-first} ($G'$ preserves $V_i$ for all $i$) and the singular locus of $\overline{V_2}/G$ is a trivalent graph with one vertex of valency $3$ by Figure \ref{fig-z1}. So the singular locus of $V/G$ has only one vertex of valency $3$, but this contradicts Proposition \ref{corollary-KJS}. In the case of $G\cong\mathbb{D}_4$, we get $G'\cong \mathbb{Z}_4$ or $\mathbb{D}_2$. In the former case, $\overline{V_1}/G'$ is $V(A0,1)$ by Lemma \ref{lemma-gm-first} ($G'$ preserves $V_i$ for all $i$.) In the latter case, $\overline{V_1}/G'$ is $V(B0,1)$ or $V(B0,2)$ by Table \ref{table-1}. In any case, the singular locus of $\overline{V_1}/G'$ has even vertices of valency $3$ and that of $\overline{V_2}/G$ is a trivalent graph with one vertex of valency $3$ by Figure \ref{fig-z1}, so that of $V/G$ has odd vertices of valency $3$. But this contradicts Proposition \ref{corollary-KJS}. In the case of $G\cong\mathbb{D}_6$, we get a contradiction similarly as the case $G\cong \mathbb{D}_4$.
\end{proof}

\begin{lemma}\label{lemma-gm-2}
Let  $G$ be a finite group of non-trivial orientation preserving diffeomorphisms which acts on a genus two handlebody $V$. Suppose that $\mathcal{S}$ is a set of properly embedded surfaces in $V$ where $\mathcal{S}$ cuts $V$ into a $3$-ball $V_1$ and a solid torus $V_2$. If $G$ preserves both $V_1$ and $V_2$, then $G\cong  \mathbb{Z}_2$ or $\mathbb{D}_2$.
\end{lemma}

\begin{proof}
Since $G$ preserves both $V_1$ and $V_2$, $V/G$ can be written as
$$V/G=\overline{V_1}/G \cup \overline{V_2}/G,$$
where $\overline{V_1}/G\cap \overline{V_2}/G$ appears only on the boundary of each quotient.\\
In the case of $G\cong \mathbb{Z}_3$, $\mathbb{Z}_4$ or $\mathbb{Z}_6$, $\overline{V_2}/G$ is $V(A0,1)$ by Lemma \ref{lemma-gm-first}. Moreover, the singular locus of $\overline{V_1}/G$ is an arc of index $k\geq 3$ by Figure \ref{fig-z1}. So the singular locus of $V/G$ consists of an arc, but this contradicts Proposition \ref{corollary-KJS}.

In the case of $G\cong \mathbb{D}_3$, $\mathbb{D}_4$ or $\mathbb{D}_6$, the singular locus of $V/G$ has odd vertices of valency $3$, (one is from $\overline{V_1}/G$ and the others are from $\overline{V_2}/G$ by Figure \ref{fig-z1} and Table \ref{table-1}) for all possible cases. But this contradicts Proposition \ref{corollary-KJS}.
\end{proof}

\begin{figure}
\includegraphics[viewport=26 643 541 780, width=11cm]{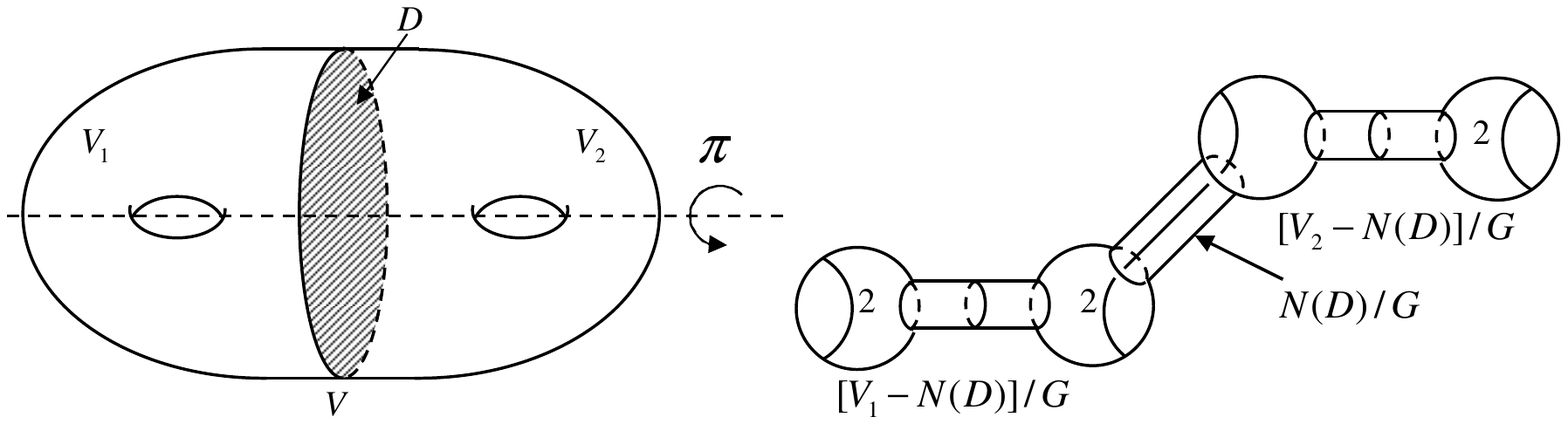}\\
\includegraphics[viewport=30 583 537 780, width=11cm]{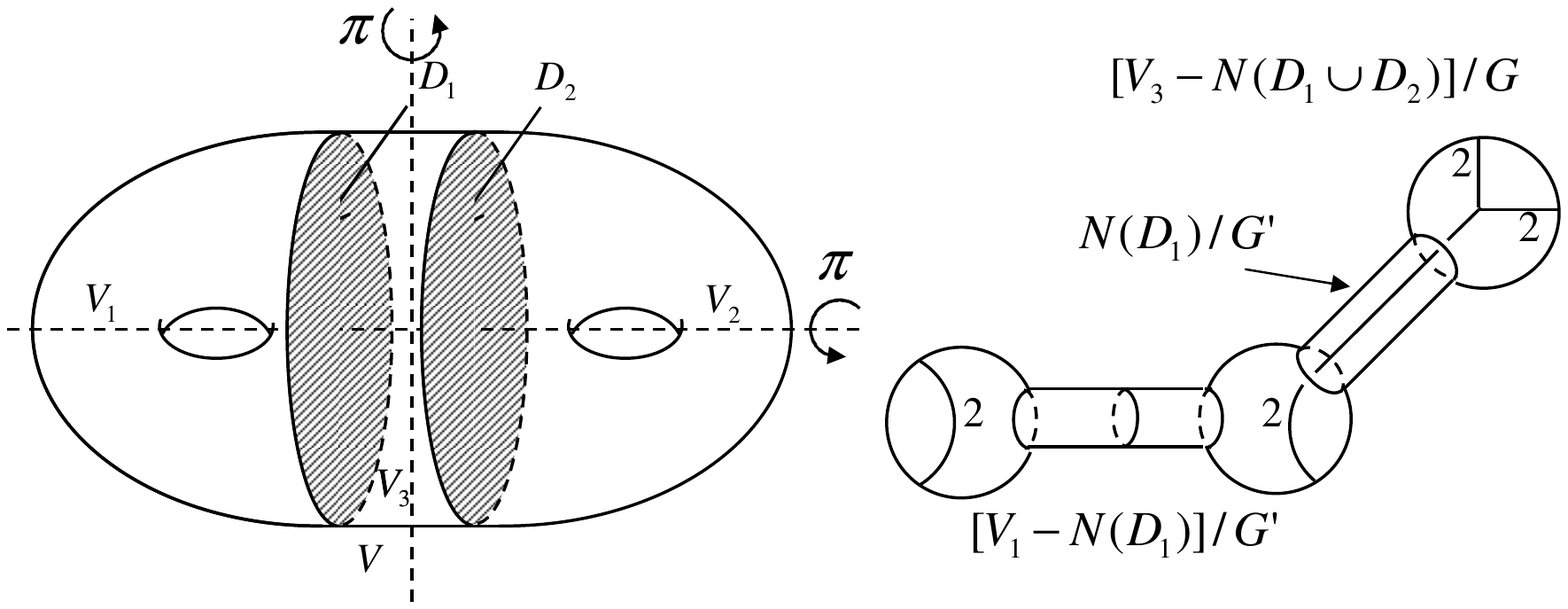}\\
\includegraphics[viewport=30 572 561 780, width=11cm]{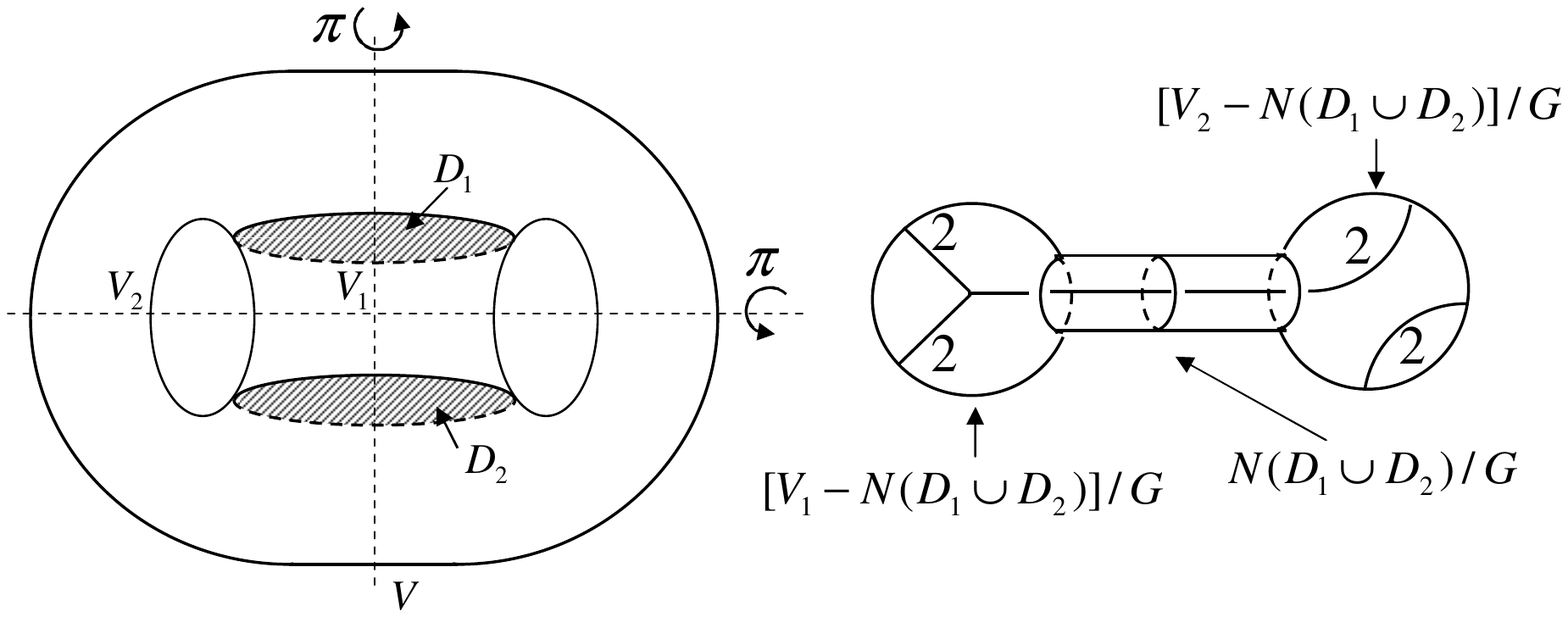}
\caption{Easy examples of Lemma \ref{lemma-gm-1}, \ref{lemma-gm-12} and \ref{lemma-gm-2}.}
\end{figure}

Let $D$ be a disk properly embedded in a handlebody $V$. $D$ is a \emph{meridian disk} of $V$ if $D$ does not separate $V$ (use the definition in p.439 of \cite{KO}.)

\begin{lemma}\label{lemma-gm-additional}
Suppose that $G$ is a finite cyclic group of non-trivial orientation preserving diffeomorphisms which acts on a genus two handlebody $V$, where $|G|\geq 3$. If $D$ is a meridian disk of $V$ where $g(D)\cap D =\emptyset$ or $D$ for all $g\in G$, then $G$ does not preserve $D$ and $G(D)$ cuts $V$ into 3-balls.
\end{lemma}

\begin{proof}
Suppose that $G$ preserves $D$. Then $G$ preserves an open regular neighborhood of $D$, say $N(D)$, in $V$. Assume that $\overline{N(D)}\cong D\times I$. So we get
$$V/G=\overline{N(D)}/G\cup [V-N(D)]/G,$$
where $\overline{N(D)}/G\cap[V-N(D)]/G = [D\times\{0,1\}]/G$.
Since $G$ is a cyclic group, $\overline{N(D)}/G$ is a $1$-handle with a core of index $|G|$ by Figure \ref{fig-z1}. Since $V-N(D)$ is a solid torus, $[V-N(D)]/G$ is $V(A0,k)$ with $k\geq 1$ by Table \ref{table-1}. So the singular locus of $V/G$ consists of an arc and at most one isolated circle. But this contradicts Proposition \ref{corollary-KJS}.

So $G$ does not preserve $D$.
Then all disks of $G(D)$ are properly embedded pairwise disjoint disks in $V$. Since $D$ is non-separating in $V$, there exist an essential simple closed curve $\alpha$ in $V$ which meets $D$ exactly once and non-tangentially. So $g(\alpha)$ also meets $g(D)$ exactly once and non-tangentially for all $g\in G$. So all disks of $G(D)$ are non-separating.

If at least two disks of $G(D)$ are non-parallel in $V$, then $G(D)$ cuts $V$ into $3$-balls, i.e. the proof ends. So we assume that all disks of $G(D)$ are parallel in $V$. Then $G(D)$ cuts $V$ into a solid torus $V^{\ast}$ and $3$-balls $B_1$, \ldots $B_n$. Since all disks $G(D)$ are parallel in $V$, there are outermost disks $D_1$ and $D_2$ in $G(D)$ which meet $\overline{V^{\ast}}$. Since $V-G(D)$ is preserved by $G$ and $V^{\ast}$ is the only solid torus component of $V-G(D)$, $V^{\ast}$ must be preserved by $G$. Since $D_1$ and $D_2$ are the only disks of $G(D)$ which meet $\overline{V^{\ast}}$, $\{D_1,D_2\}$ is also preserved by $G$. So $V/G$ can be written as
$$V/G=\overline{V^{\ast}}/G\cup_{\{D_1,D_2\}/G}\{\overline{B_1}\cup\ldots\cup\overline{B_n}\}/G.$$
Since $\overline{V^{\ast}}/G$ is $V(A0,1)$ by Lemma \ref{lemma-gm-first} and $\{\overline{B_1}\cup\ldots\cup\overline{B_n}\}/G$ connects $\overline{V^{\ast}}/G$ only through $\{D_1,D_2\}/G$, the base space of $V/G$ cannot be a $3$-ball. But this contradicts Proposition \ref{corollary-KJS}.
\end{proof}

\begin{lemma}\label{lemma-gm-an-1}
Let $G$ be a finite group of non-trivial orientation preserving diffeomorphisms which acts on a genus two handlebody $V$. Suppose that $\mathcal{S}$ is a set of properly embedded surfaces in $V$ where $\mathcal{S}$ cuts $V$ into two connected components $V_1$ and $V_2$, where both $V_1$ and $V_2$ are noncontractible in $V$, and $V_1$ or $V_2$ contains a meridian disk of $V$. If $G$ preserves both $V_1$ and $V_2$, then $G\cong  \mathbb{Z}_2$ or $\mathbb{D}_2$.
\end{lemma}

\begin{proof}
Let the meridian disk of $V$ be $D$. We can assume that $g(D)\cap D=\emptyset$ or $D$ for all $g\in G$ by the equivariant loop theorem/Dehn lemma (see \cite{MY}.)
In the case of $G\cong \mathbb{Z}_3$, $\mathbb{Z}_4$ or $\mathbb{Z}_6$, $V-G(D)$ consists of $3$-balls by Lemma \ref{lemma-gm-additional}, where $G(D)\subset V_1$ or $G(D)\subset V_2$ ($G$ preserves both $V_1$ and $V_2$.) Then we get $\operatorname{int}(V_2) \subset V-G(D)$ or $\operatorname{int}(V_1)\subset V-G(D)$, i.e. $V_2$ or $V_1$ is contractible in $V$, this contradicts the hypothesis.

In the case of $G\cong \mathbb{D}_3$, $\mathbb{D}_4$ or $\mathbb{D}_6$, $G$ must have an index 2 cyclic subgroup $\bar{G}$. But $|\bar{G}|\geq3$ and $\bar{G}$ also preserves $V_i$ for $i=1,2$, we get a contradiction using the same arguments in the case of $G\cong \mathbb{Z}_3$, $\mathbb{Z}_4$ or $\mathbb{Z}_6$.
\end{proof}

\begin{lemma}\label{lemma-gm-an-2}
Let  $G$ be a finite group of non-trivial orientation preserving diffeomorphisms which acts on a genus two handlebody $V$. Suppose that $\mathcal{S}$ is a set of properly embedded surfaces in $V$ where $\mathcal{S}$ cuts $V$ into two solid tori $V_1$ and $V_3$, and a genus two handlebody $V_2$, where
    \begin{enumerate}
        \item $V'=\overline{V_1}\cup \overline{V_2}$ is a genus two handlebody,
        \item both $V_1$ and $V_2$ are non-contractible in $V'$ and $V$, and
        \item there is a meridian disk $D$ in $V$ such that $D\cap V_1=\emptyset$ and $D\cap V_2$ is a non-empty set of meridian disks of $V'$.
    \end{enumerate}
If $G$ preserves $V_i$ for $i=1,2,3$, then $G\cong  \mathbb{Z}_2$ or $\mathbb{D}_2$. If $G$ preserves $V_2$ and exchanges $V_1$ and $V_3$, then $G\cong \mathbb{Z}_2$ or $\mathbb{D}_2$.
\end{lemma}

\begin{proof}
We can assume that $g(D)\cap D=\emptyset$ or $D$ for all $g\in G$ by the equivariant loop theorem/Dehn lemma.

Suppose that $G$ preserves each of $V_i$, $i=1,2,3.$ Then $G$ preserves the genus two handlebody $V'$. Since $V'$ holds the hypothesis of Lemma \ref{lemma-gm-an-1} by the decomposition of $V_1$ and $V_2$, we get $G\cong \mathbb{Z}_2$ or $\mathbb{D}_2$.

Suppose that $G$ preserves $V_2$ and exchanges $V_1$ and $V_3$. Then $V/G$ can be written as
$$V/G=\overline{V_1}/G' \cup \overline{V_2}/G,$$
where $G'$ is an index two subgroup of $G$ which preserves both $V_1$ and $V_3$, and $\overline{V_1}/G'\cap \overline{V_2}/G$ appears only on the boundary of each quotient.

Since $V'=\overline{V_1}\cup \overline{V_2}$ holds the hypothesis of Lemma \ref{lemma-gm-an-1} by the decomposition of $V_1$ and $V_2$ with the group $G'$, we get $G'\cong \mathbb{Z}_2$ or  $\mathbb{D}_2$ (if $G'$ is nontrivial), or $G'$ is trivial. So $G\cong \mathbb{Z}_2$, $\mathbb{Z}_4$, $\mathbb{D}_2$ or $\mathbb{D}_{4}$.

Suppose that $G\cong \mathbb{Z}_4$. Then $V/G$ is a $3$-ball with singular locus of two arcs, where the indices are $4$ and $2$ by Proposition \ref{corollary-KJS}. In particular, the arc of index $4$ in $V/G$ must be induced from the set of fixed points of $G$ from $|G|=4$. By Lemma \ref{lemma-gm-additional}, $G(D)$ cuts $V$ into $3$-balls. Let $A$ be the set of the fixed points of $G$ in $V$ and $B$ be the ball which contains $A$ among $V-G(D)$. Obviously, $G$ preserves $B$. Since $G$ is a finite cyclic group action on a $3$-ball $B$, the generator of $G$, say $\gamma$, can be realized by a $\pi/2$-rotation $\gamma$ about $A$ on $B$.

\begin{figure}
\centering
\includegraphics[viewport=28 432 566 780,width=8cm]{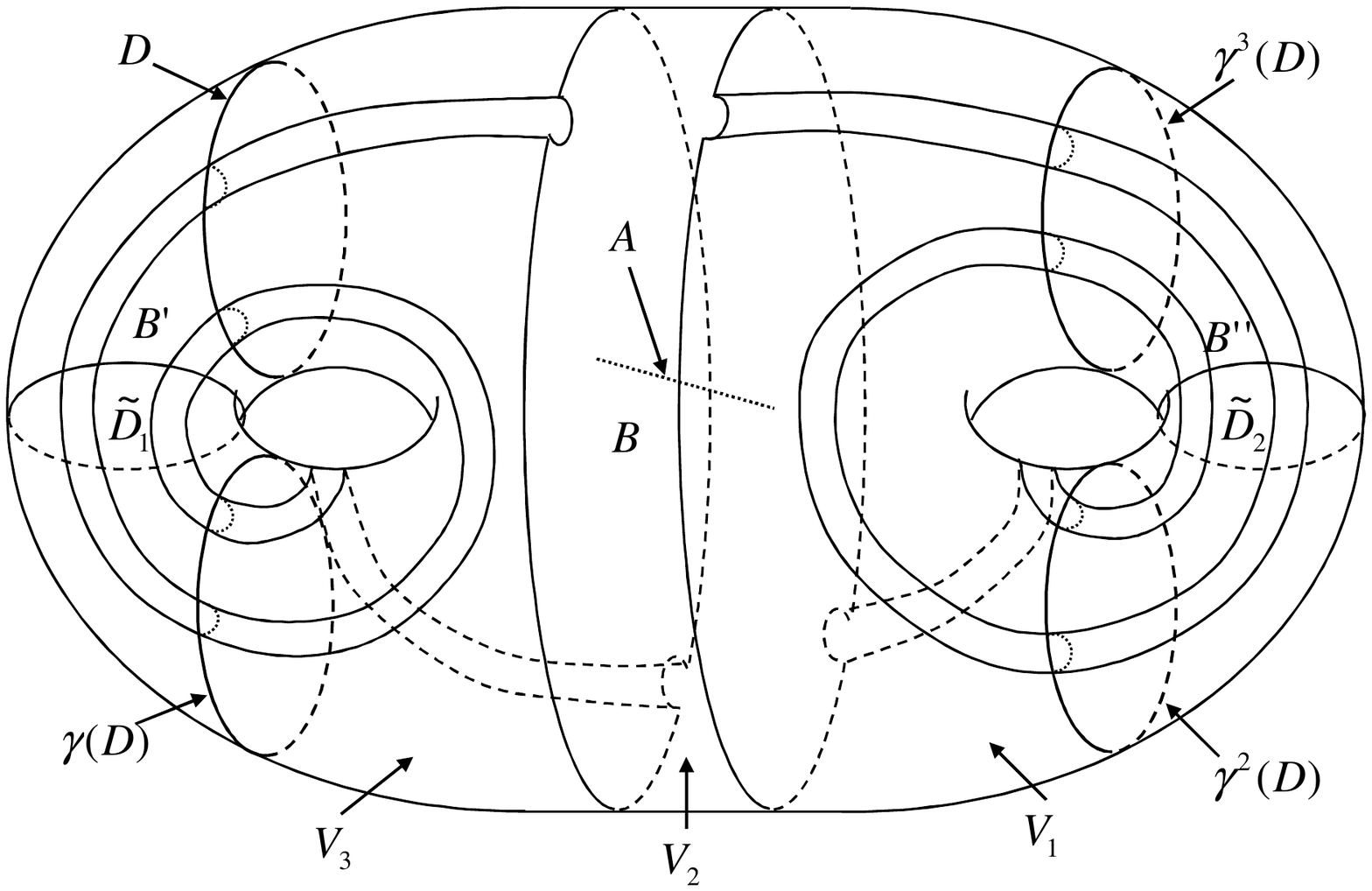}
\caption{$G(D)$ cuts $V$ into three $3$-balls.\label{figure-gm-add2}}
\end{figure}

Figure \ref{figure-gm-add2} is a good example. In this figure, $G(D)$ cuts $V$ into three $3$-balls $B$, $B'$ and $B''$. $G(D)$ consists of four disks, and $G(\tilde{D}_1)$ consists of two disks $\tilde{D}_1$ and $\tilde{D}_2$. The arc of index $2$ in the singular locus of $V/G$ is from the quotient $(\tilde{D}_1\cup\tilde{D}_2)/G$.

So there exist a $G$-invariant open regular neighborhood $N(A)$ of $A$ in $B$ such that $\overline{N(A)}\simeq D\times I$,  $D\times\{0,1\}\subset\partial B$, $G$ preserves each level $D\times \{i\}$ for $0\leq i \leq 1$, and $A$ is the core of $D\times I$.

Since the singular locus of $V/G$ has only one arc of index $4$ and $G$ also acts on a genus two handlebody $V_2\subset V$, both $V$ and $V_2$ have the same fixed point set if we use Proposition \ref{corollary-KJS} again. Therefore, we can assume that $N(A)\subset V_2$, i.e. both $V_1$ and $V_3$ do not meet $N(A)$. So we can draw a top down view of $V_2$ near $N(A)$ as Figure \ref{figure-gm-a1}. Since the $\pi/2$-rotation $\gamma$ exchanges $V_1$ and $V_3$, $V_1$ (so $V_3$) is not contractible in $V$, and there is a meridian disk $D$ in $V$ which does not meet $V_1$, $V_1$ is depicted as a strip from the upper side to the lower side and $V_3$ is depicted as a strip from the left side to the right side (so the upper side extends outside the dotted circle and will connect the lower side, and the left side extends outside the dotted circle and will connect the right side in the top down view.) But $\gamma$ cannot exchange $V_1$ and $V_3$ from Figure \ref{figure-gm-a1} unless both $V_1$ and $V_3$ meet the core $A$. So we get a contradiction.

\begin{figure}
\centering
\includegraphics[viewport=28 494 315 781,width=6cm]{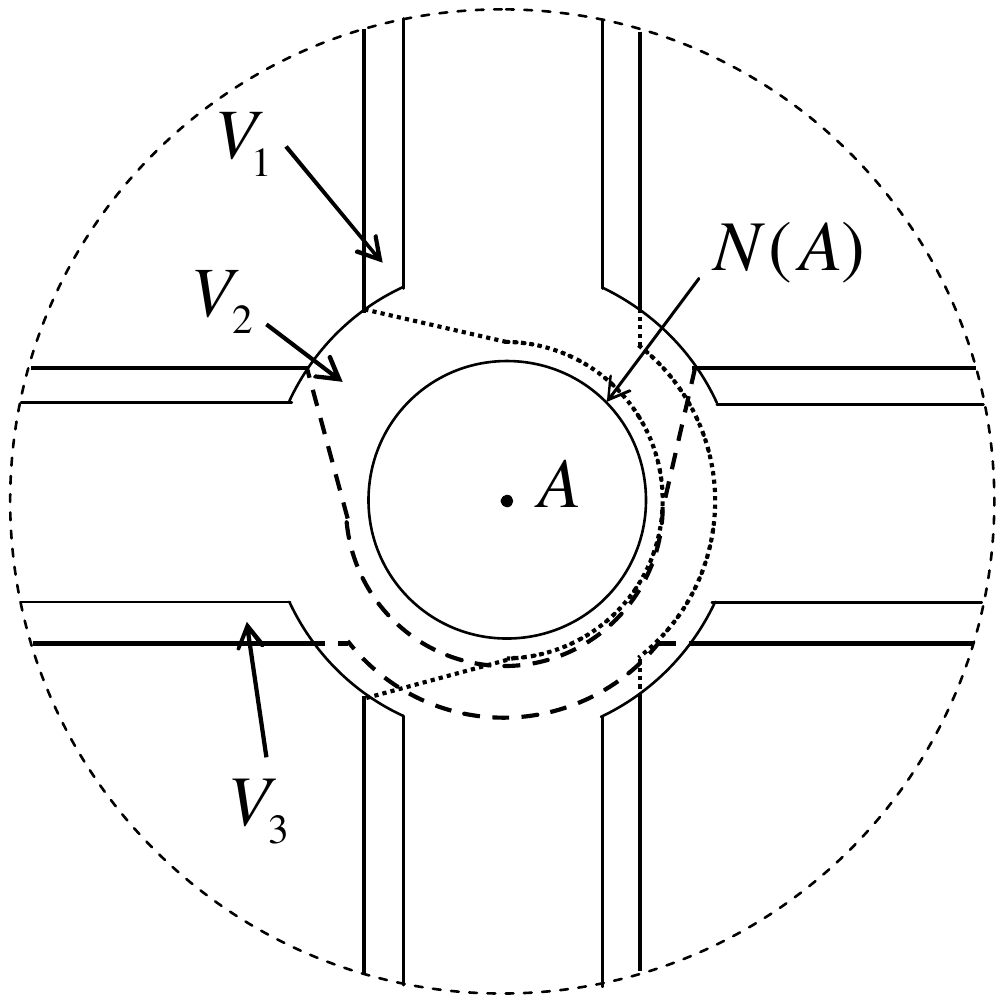}
\caption{A local figure near $N(A)$. In this figure both $V_1$ and $V_3$ do not meet the core $A$.\label{figure-gm-a1}}
\end{figure}

Suppose that $G\cong \mathbb{D}_4$, then $G$ must have an index 2 cyclic subgroup $\bar{G}\cong \mathbb{Z}_4$. Then $\bar{G}$ preserves each $V_i$ for $i=1,2,3$, or preserves $V_2$ and exchanges $V_1$ and $V_3$. In each case, we get a contradiction by the previous arguments.

\end{proof}

\section{Proof of Theorem \ref{main-theorem} in the case of Condition A with one separating JSJ-torus.\label{section2-1}}
In this section, we will prove Theorem \ref{main-theorem} when the JSJ-tori and the Heegaard splitting satisfy Condition A in the one separating JSJ-torus case.

Let $T$ be the JSJ-torus. By this JSJ-decomposition, we will denote that $M=M_1 \cup_T M_2$  where each $M_i$ is a piece of the JSJ-decomposition for $i=1,2$.

At first, we consider the case when $T\cap V_1$ consists of minimal number of disks.

\begin{proposition}(Case 2 of ``proof of Theorem'' in p.449 \cite{KO}) \label{proposition-1} Suppose that $M$ has only one separating JSJ-torus. If $T\cap V_1$ consists of disks and the number of disks is minimal, then $T\cap V_1$ consists of one disk or two mutually disjoint disks, i.e. the Heegaard splitting and the JSJ-tori satisfy Condition A.
\end{proposition}

Let $T_2$ be $T\cap V_2$. Then as in p.21 of \cite{JA}, we have a hierarchy $(T_2^{(0)},a_0)$, $\cdots$, $(T_2^{(m)},a_m)$ of $T_2$ and a sequence of isotopies of type A which realizes the hierarchy where each $a_i$ is an essential arc in $T_2^{(i)}$ (see p.449 of \cite{KO}).
Let $T^{(1)}$ be the image of $T$ after an isotopy of type A at $a_0$ and $T^{(k+1)}$ ($ k\geq 1$) be the image of $T^{(k)}$ after an isotopy of type A at $a_k$.

By Proposition \ref{proposition-1}, the proof is divided into the two cases.\\

\Case{1} $T\cap V_1$ consists of a disk $\bar{D}$. Since $T$ separates $M$, $\bar{D}$ separates $V_1$ into two solid tori $V_1^1\subset M_1$ and $V_1^2\subset M_2$ (this case is exactly the same as Case 2.1 in p.451 of \cite{KO}, see Figure \ref{figure-1}.)

\begin{figure}
\centering
\includegraphics[viewport=27 428 550 780,width=7cm]{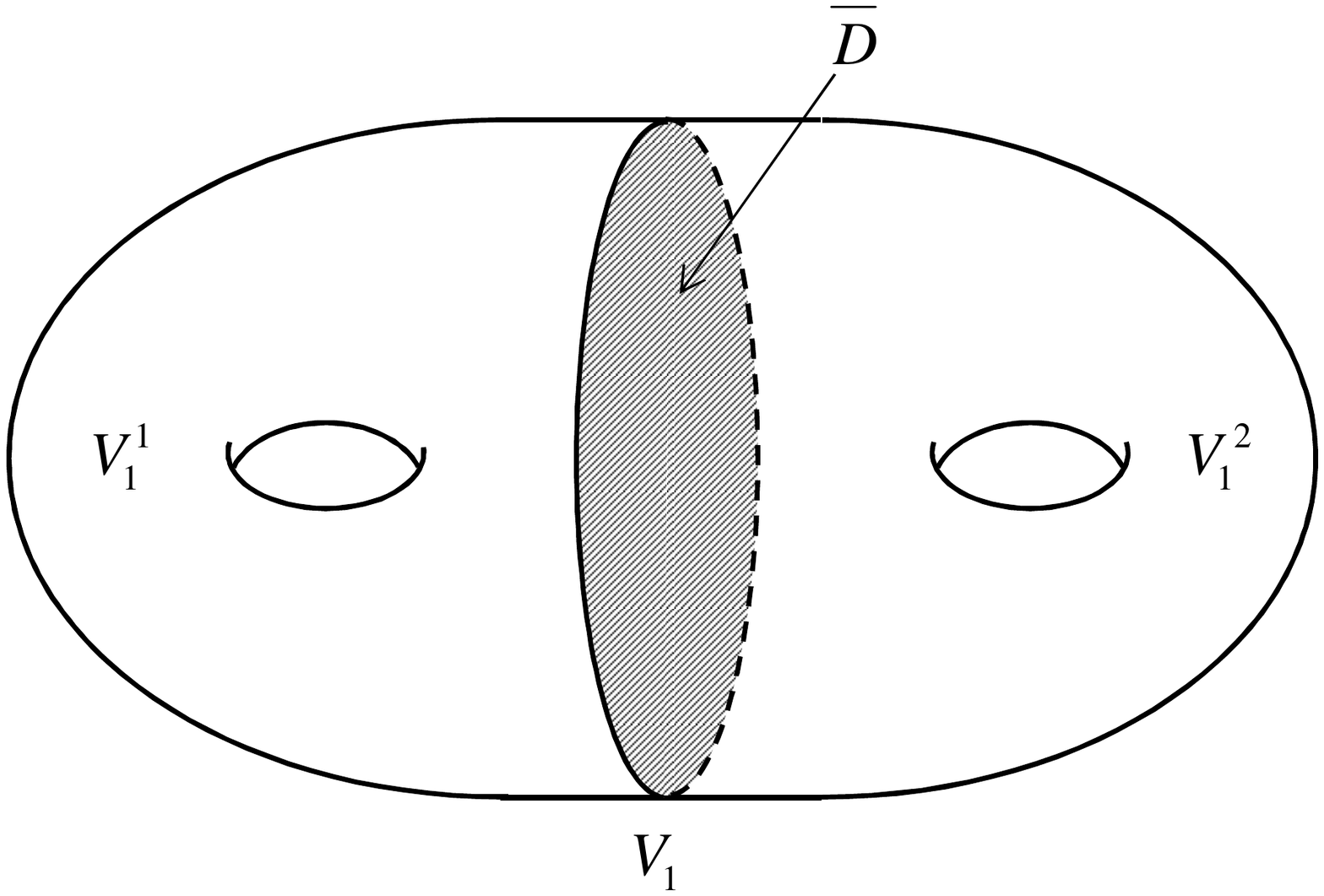}
\caption{In the case when $T\cap V_1$ consists of a disk
$\bar{D}$.\label{figure-1}}
\end{figure}

Since $G$ preserves each piece of the JSJ-decomposition, it also preserves both $V_1^1$ and $V_1^2$. So $G\cong \mathbb{Z}_2$ by Lemma \ref{lemma-gm-1}. This completes the proof of Case 1.\\

\Case{2}
$T\cap V_1$ consists of two disks $\bar{D}_1$ and $\bar{D}_2$ (this case is exactly the same as Case 2.2 in p.451 of \cite{KO}.)

In this case, $T^{(2)}\cap V_1$ consists of two essential annuli $A_1$ and $A_2$ in $V_1$ (see p.451 of \cite{KO}.) In order to know the possible positions of $A_1$ and $A_2$, we need the following lemma.

\begin{lemma}\label{lemma-ko-3-4}(\cite{KO}, Lemma 3.4) Let $\{A_1, A_2\}$ be a system of mutually disjoint, non-parallel, essential annuli in a genus two handlebody $V$. Then either
    \begin{enumerate}[(i)]
        \item \label{lemma-ko-3-4-1}$A_1\cup A_2$ cuts $V$ into a  solid torus $V_1$ and a genus two handlebody $V_2$. Then $A_1\cup A_2\subset \partial V_1$, $A_1\cup A_2\subset\partial V_2$ and there is a complete system of meridian disks $\{D_1,D_2\}$ of $V_2$ such that $D_i\cap A_j=\emptyset$ ($i\neq j$) and $D_i\cap A_i$ ($i=1,2$) is an essential arc of $A_i$,
        \item \label{lemma-ko-3-4-2}$A_1\cup A_2$ cuts $V$ into two solid tori $V_1$, $V_2$ and a genus two handlebody $V_3$. Then $A_1\subset \partial V_1$, $A_2\subset \partial V_2$, $A_1\cup A_2 \subset \partial V_3$ and there is a complete system of meridian disks $\{D_1, D_2\}$ of $V_3$ such that $D_i\cap A_j=\emptyset$ $(i\neq j)$ and $D_i\cap A_i$ ($i=1,2$) is an essential arc of $A_i$ or
        \item \label{lemma-ko-3-4-3}$A_1\cup A_2$ cuts $V$ into a solid torus $V_1$ and a genus two handlebody $V_2$. Then $A_i\subset\partial V_1$ ($i=1$ or $2$, say $1$), $A_2\cap V_1=\emptyset$, $A_1\subset\partial V_2$ and there is a complete system of meridian disks $\{D_1, D_2\}$ of $V_2$ such that $D_1\cap A_2$ is an essential arc of $A_2$ and $D_2\cap            A_i$ ($i=1,2$) is an essential arc of $A_i$, see Figure \ref{fig-lemma-ko-3-4}.
    \end{enumerate}
\end{lemma}

\begin{figure}
\centering
\includegraphics[viewport=28 144 336 777, width=8cm]{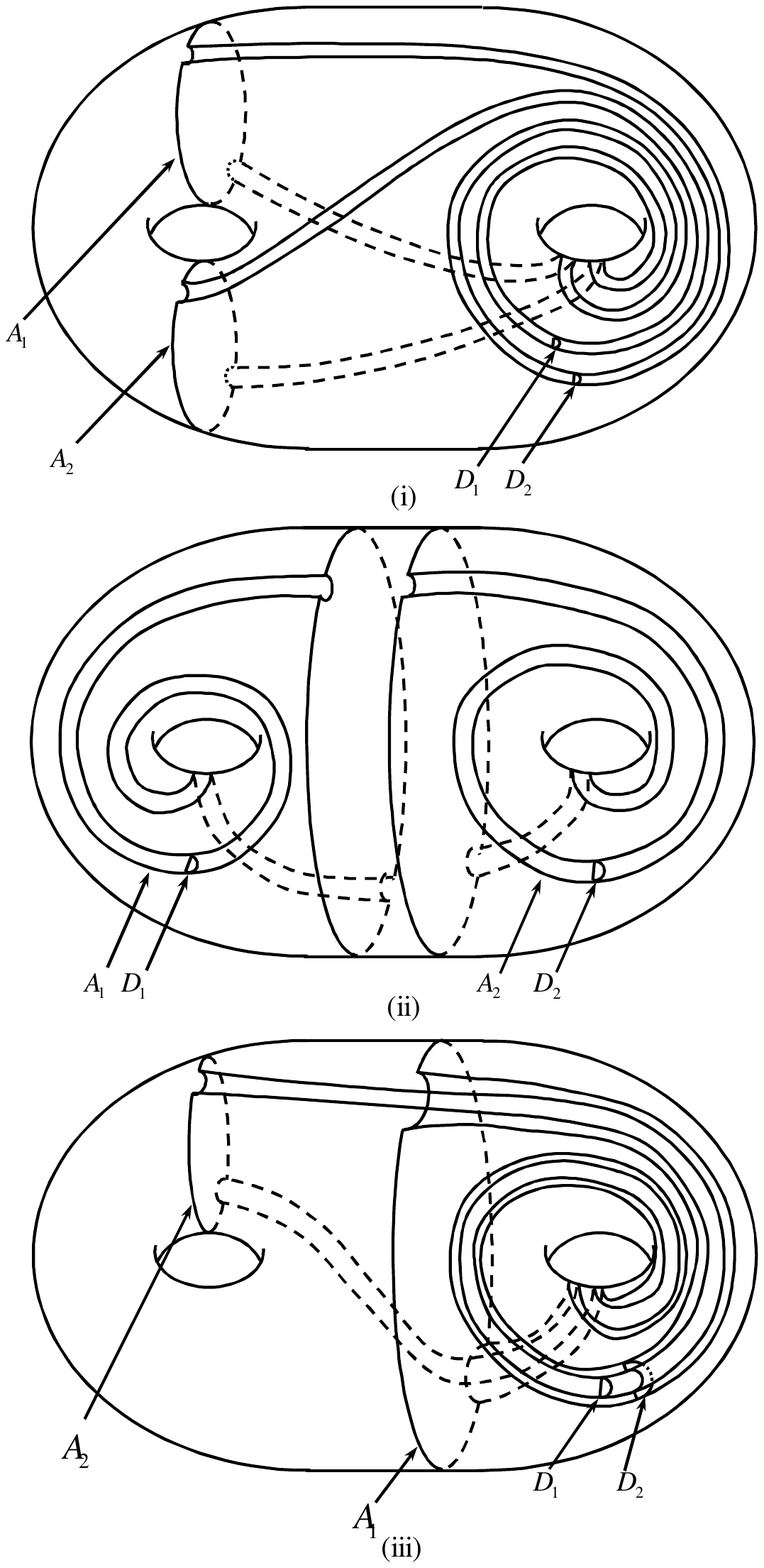}
\caption{The possibilities of two mutually disjoint non-parallel essential annuli in the genus two handlebody. \cite{KO}\label{fig-lemma-ko-3-4}}
\end{figure}

If $A_1$ and $A_2$ are parallel in $V_1$, then then $\{A_1, A_2\}$ satisfies the conclusion (\ref{lemma-ko-3-4-1}) of Lemma \ref{lemma-ko-3-4} (see the claim of Case 2.2 in p.451 of \cite{KO}.) If $A_1$ and $A_2$ are non-parallel in $V_1$, then $\{A_1, A_2\}$ satisfies one of the conclusions of Lemma \ref{lemma-ko-3-4}. Moreover, $\{A_1, A_2\}$ does not satisfy the conclusion (\ref{lemma-ko-3-4-3}) by the argument of the third paragraph of p.452 in \cite{KO}.\\

\Claim{$\bar{D}_1$ and $\bar{D}_2$ are parallel in $V_1$.
If each $\bar{D}_i$ is non-separating in $V_1$ for $i=1,2$, then $\{\bar{D}_1,\,\bar{D}_2\}$ cuts $V_1$ into two parts where one is a $3$-ball and the other  is a solid torus. If each $\bar{D}_i$ is separating in $V_1$ for $i=1,2$, then $\{\bar{D}_1,\,\bar{D}_2\}$ cuts $V_1$ into three parts where one is a $3$-ball and the others are solid tori.}

\begin{proofc}
We already know that $\{A_1,A_2\}$ satisfies the conclusion (\ref{lemma-ko-3-4-1}) or (\ref{lemma-ko-3-4-2}) of Lemma \ref{lemma-ko-3-4}. As the disks $\{D', D''\}$ in the proof of Lemma \ref{lemma-ko-3-4}, each $\bar{D}_i$ is obtained by a surgery on $A_i$ along a disk $\Delta_i$ such that $\Delta_i\cap A_i=a_i$ is an essential arc of $A_i$, $\Delta_i \cap \partial V_1=b_i$ is an arc in $\partial \Delta_i$, $a_i\cap b_i = \partial a_i \cup \partial b_i$ and $a_i\cup b_i=\partial\Delta_i$ for $i=1,2$. So if we check Case 1 (corresponding to (\ref{lemma-ko-3-4-1}) of Lemma \ref{lemma-ko-3-4}) and Case 2 (corresponding to (\ref{lemma-ko-3-4-2}) of Lemma \ref{lemma-ko-3-4}) in the proof of Lemma \ref{lemma-ko-3-4}, then we can see that $\bar{D}_1$ and $\bar{D}_2$ are parallel in $V_1$ (Case 3 in the proof of Lemma \ref{lemma-ko-3-4} corresponds to the conclusion (\ref{lemma-ko-3-4-3}) of Lemma \ref{lemma-ko-3-4}.) If $\bar{D}_1$ and $\bar{D}_2$ are non-separating in $V_1$, then we get the result after the surgery identifying $\Delta_i$ with $D_i$ in (i) of Figure \ref{fig-lemma-ko-3-4}. If $\bar{D}_1$ and $\bar{D}_2$ are separating in $V_1$, then we get the result after the surgery identifying $\Delta_i$ with $D_i$ in (ii) of Figure \ref{fig-lemma-ko-3-4}.
\end{proofc}

By Claim, $\{\bar{D}_1,\bar{D}_2\}$ cuts $V_1$ into two or three parts.

If $\{\bar{D}_1,\bar{D}_2\}$ cuts $V_1$ into two parts, then one is a $3$-ball $V_1^1$ and the other is a solid torus $V_1^2$ (see Figure \ref{figure-4}.)

\begin{figure}
\centering
\includegraphics[viewport=25 464 554 780,width=8cm]{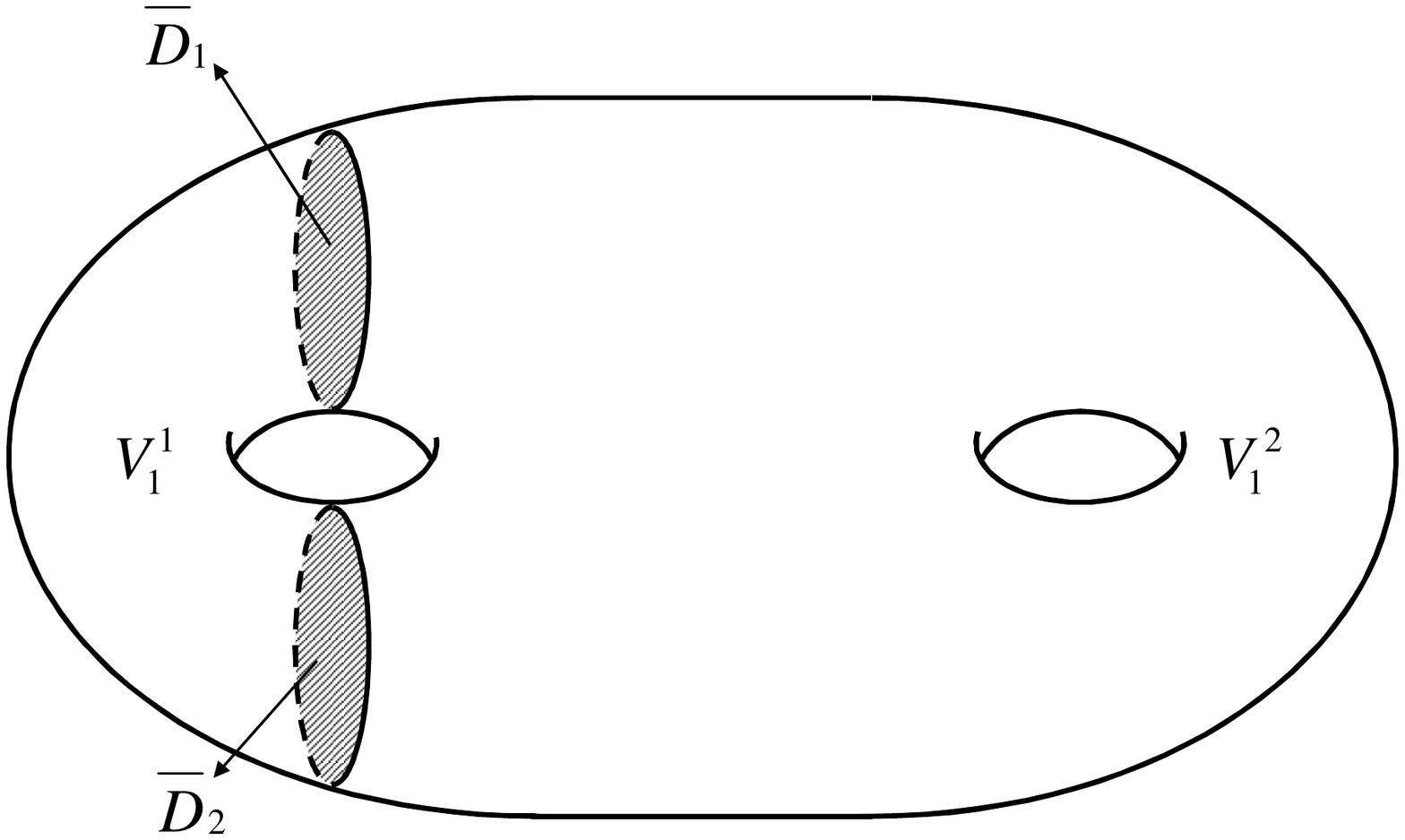}
\caption{$\bar{D}_1$ and $\bar{D}_2$ (non-separating disks)\label{figure-4}}
\end{figure}

Since $G$ preserves each $M_i$ for $i=1,2$ and $V_1$, it also preserves both $V_1^1$ and $V_1^2$. So $G\cong\mathbb{Z}_2$ or $\mathbb{D}_2$ by Lemma \ref{lemma-gm-2}.

If $\{\bar{D}_1,\bar{D}_2\}$ cuts $V_1$ into three parts, then one is a $3$-ball $V_1^2$ and the others are two solid tori $V_1^1$ and $V_1^3$ (see Figure \ref{figure-8}.) In this case, $V_1^1$ and $V_1^3$ are contained in $M_1$ (or $M_2$) and $V_1^2$ is contained in $M_2$ (or $M_1$) ($T$ separates $M$ into two pieces $M_1$ and $M_2$.) So $G$ preserves each $V_1^i$ for $i=1,2,3$, or $G$ preserves $V_1^2$ and exchanges $V_1^1$ and $V_1^3$. Hence we get $G\cong \mathbb{Z}_2$ or $\mathbb{D}_2$ by Lemma \ref{lemma-gm-12}. This completes the proof of Case 2.

\begin{figure}
\centering
\includegraphics[viewport=25 423 552 779,width=8cm]{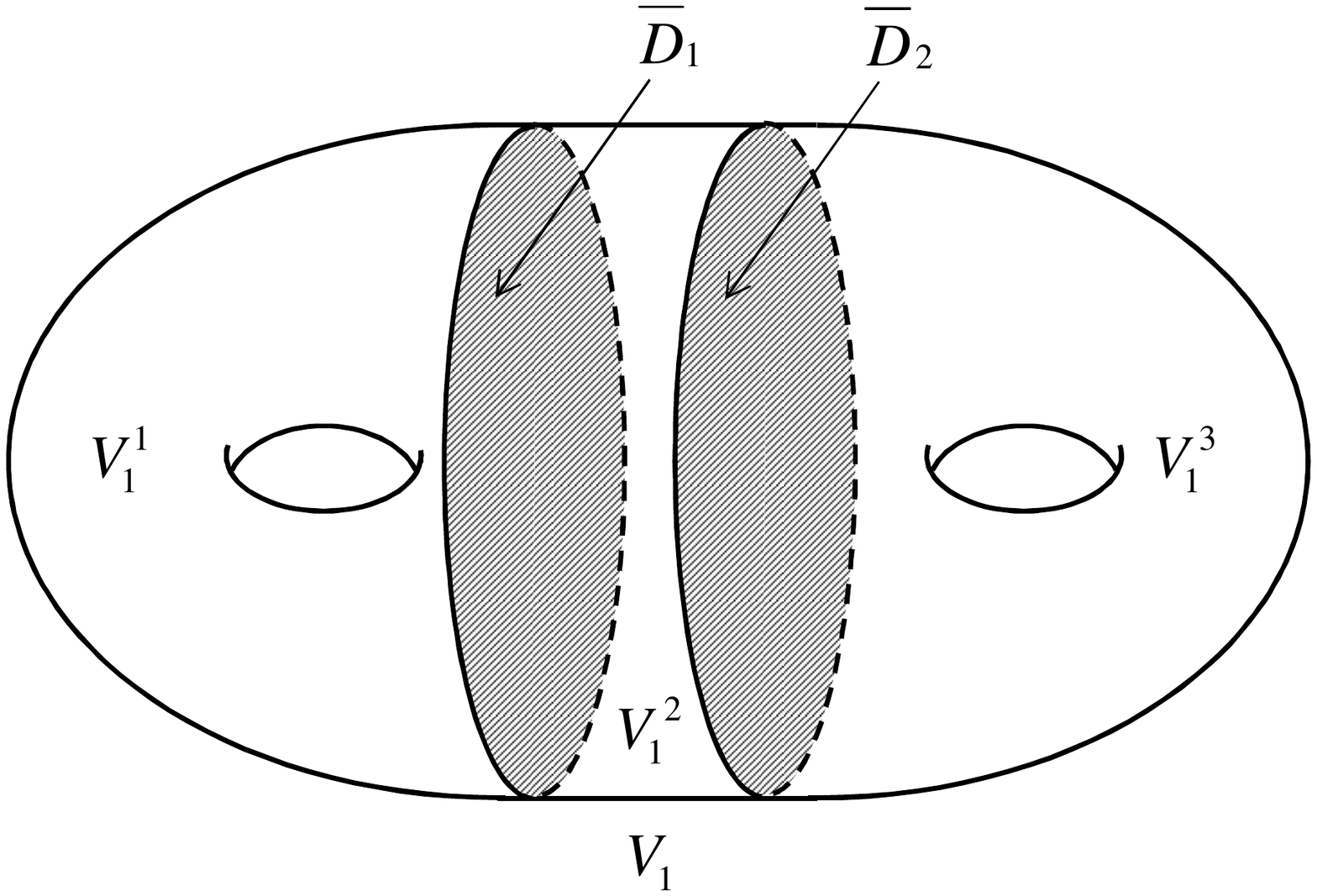}
\caption{$\bar{D}_1$ and $\bar{D}_2$ (separating disks)\label{figure-8}}
\end{figure}

Next, we consider the case when $T\cap V_1$ consists of non-minimal number of disks.
Since Condition A assumes that $T\cap V_1$ consists of at most two disks and the disks are not $\partial$-parallel in $V_1$, we only need to check the case when $T\cap V_1$ consists of two essential disks $\bar{D}_1$ and $\bar{D}_2$ in general situation.

\Case{3} Both $\bar{D}_1$ and $\bar{D}_2$ are separating in $V_1$ (so $\bar{D}_1$ and $\bar{D}_2$ are parallel in $V_1$.) Then $\{\bar{D}_1,\bar{D}_2\}$ cuts $V_1$ into three parts, one is a $3$-ball and the others are two solid tori. So we get $G\cong \mathbb{Z}_2$ or $\mathbb{D}_2$ by similar arguments in the three-parts case of Case 2.

\Case{4} Both $\bar{D}_1$ and $\bar{D}_2$ are non-separating in $V_1$. If they are parallel in $V_1$, then $\{\bar{D}_1,\bar{D}_2\}$ cuts $V_1$ into two parts, one is a $3$-ball and the other is a solid torus. So we get $G\cong \mathbb{Z}_2$ or $\mathbb{D}_2$ by similar arguments in the two-parts case of Case 2. If they are non-parallel in $V_1$, then $\{\bar{D}_1,\bar{D}_2\}$ cuts $V_1$ into a $3$-ball. But this means that $V_1$ consists of only one piece of the JSJ-decomposition, so contradicts the fact that $T$ separates $M$ into two pieces.

\Case{5} $\bar{D}_1$ is separating in $V_1$ and $\bar{D}_2$ is non-separating in $V_1$. But this contradicts the fact that $\bar{D}_2$ is a part of the JSJ-torus $T$ ($T$ separates $M$ into two pieces.)

Now this completes the proof of Theorem \ref{main-theorem} when the JSJ-tori and the Heegaard splitting satisfy Condition A in the one separating JSJ-torus case.

\section{Proof of Theorem \ref{main-theorem} in the case of Condition A with two separating JSJ-tori\label{section2-2}}
In this section, we will prove Theorem \ref{main-theorem} when the JSJ-tori and the Heegaard splitting satisfy Condition A in the two separating JSJ-tori case.

Let $T_1$ and $T_2$ be the JSJ-tori. By this JSJ-decomposition, we will denote that $M=M_1 \cup_{T_1} M_2 \cup_{T_2}M_3$ where each $M_i$ is a piece of the JSJ-decomposition for $i=1,2,3$.

Let $\mathcal{T}$ be $T_1\cup T_2$ and $\mathcal{T}'$ be $\mathcal{T}\cap V_2$.

\begin{proposition}(Case 3 in pp.452--453 of \cite{KO})\label{proposition-5}
Suppose that $M$ has two separating tori. If $\mathcal{T}\cap V_1$ consists of disks and the number of disks is minimal, then $\mathcal{T}\cap V_1$ consists of two mutually disjoint disks, i.e. the Heegaard splitting and the JSJ-tori satisfy Condition A. So $T_i\cap V_1$ is a disk $\bar{D}_i$ for $i=1,2$.
\end{proposition}

Note that a JSJ-torus must intersect $V_1$ since it is incompressible in $M$ (see Lemma 2.2 of \cite{KO3}.) So we only need to check this minimal components case for Condition A.

By Proposition \ref{proposition-5}, we get the possible situation when the JSJ-tori $\mathcal{T}$ and the Heegaard splitting satisfy Condition A. Now we have a hierarchy $(\mathcal{T}'^{(0)},a_0)$, $\cdots$, $(\mathcal{T}'^{(m)},a_m)$ of $\mathcal{T}'$ and a sequence of isotopies of type A which realizes the hierarchy where $a_i$ is an essential arc in $\mathcal{T}'^{(i)}$.

Let $\mathcal{T}^{(1)}$ be the image of $\mathcal{T}$ after an isotopy of type A at $a_0$, $\mathcal{T}^{(2)}$ be the image of $\mathcal{T}^{(1)}$ after an isotopy of type A at $a_1$. Then $\mathcal{T}^{(2)}\cap V_1$ consists of two essential annuli $A_1$ and $A_2$ (see the second paragraph in p.453 of \cite{KO}.)

Since $T_1$ and $T_2$ are separating JSJ-tori in $M$, $\{\bar{D}_1,\bar{D}_2\}$ cuts $V_1$ into three parts, say $V_1^1\subset M_1$, $V_1^2\subset M_2$ and $V_1^3\subset M_3$.\\

\Claim{
$\bar{D}_1$ and $\bar{D}_2$ are parallel and each $\bar{D}_i$ is a separating disk in $V_1$ for $i=1,2$. Moreover, $\{\bar{D}_1,\bar{D}_2\}$ cuts $V_1$ into three parts where one is a $3$-ball and the others are solid tori.}

\begin{proofc}
In this case, $\{A_1, A_2\}$ satisfies the conclusion (\ref{lemma-ko-3-4-2}) of Lemma \ref{lemma-ko-3-4} (see the third paragraph in p.453 of \cite{KO}.) So if we use the same arguments in the proof of Claim of Case 2 in the previous section, then we get the result.
This completes the proof of Claim.
\end{proofc}

By Claim, $\{\bar{D}_1,\,\bar{D}_2\}$ cuts $V_1$ into three parts, the $3$-ball $V_1^2$ and the solid tori $V_1^1$ and $V_1^3$ (see Figure \ref{figure-8}.) Since $G$ preserves each $M_i$ for $i=1,2,3$, it also preserves each $V_1^i$ for $i=1,2,3$. So by Lemma \ref{lemma-gm-12}, $G\cong\mathbb{Z}_2$ or $\mathbb{D}_2$. This completes the proof of Theorem \ref{main-theorem} when the JSJ-tori and the Heegaard splitting satisfy Condition A in the two separating JSJ-tori case.

\section{Proof of Theorem \ref{main-theorem} in the case of Condition A with two non-separating JSJ-tori.\label{section2-3}}

In this section, we will prove Theorem \ref{main-theorem} when the JSJ-tori and the Heegaard splitting satisfy Condition A in the two non-separating JSJ-tori case.

Let $T_1$ and $T_2$ be the non-separating JSJ tori in $M$ and $\mathcal{T}$ be $T_1\cup T_2$.
By this JSJ-decomposition, we will denote that $M=M_1 \cup_{T_1, T_2} M_2$ where each $M_i$ is a piece of the JSJ-decomposition for $i=1,2$. If the Heegaard splitting and the JSJ-tori satisfy Condition A, then $T_i\cap V_1$ is a disk for $i=1,2$. In particular, T. Kobayashi proved that there exists a sequence of ambient isotopies of $\mathcal{T}$ such that $T_1\cap V_1$ consists of a meridian disk after the isotopies in the proof of Theorem 1 of \cite{KO2}.

\begin{lemma}\label{lemma-thm-3}
Suppose that $M$ has two non-separating JSJ-tori. If the Heegaard splitting and the JSJ-tori satisfy Condition A, then $\mathcal{T}\cap V_1$ consists of two mutually disjoint, parallel, and non-separating disks in $V_1$.
\end{lemma}

\begin{proof}
Suppose that the Heegaard splitting and the JSJ-tori satisfy Condition A, i.e. $T_i\cap V_1=\bar{D}_i$ for $i=1,2$.\\

\Claim{
Both $\bar{D}_1$ and $\bar{D}_2$ are non-separating in $V_1$.
}

\begin{proofc}
The proof is exactly the same as the proof of Lemma 4.2 of \cite{KO2}.
\end{proofc}

By Claim, it suffices to show that $\bar{D}_1$ and $\bar{D}_2$ are parallel in $V_1$.

Suppose that $\bar{D}_1$ and $\bar{D}_2$ are non-parallel in $V_1$.  Then $\{\bar{D}_1,\bar{D}_2\}$ cuts $V_1$ into a $3$-ball, i.e. the result is connected. Since all components of $V_2-\mathcal{T}$ must meet $V_1-\{\bar{D}_1,\bar{D}_2\}$, $M-\mathcal{T}$ is connected (whether $V_2-\mathcal{T}$ is connected or not.) But the set of JSJ-tori must divide $M$ into several pieces in this case, we get a contradiction.
\end{proof}

By Lemma \ref{lemma-thm-3}, $\bar{D}_1$ and $\bar{D}_2$ are non-separating and parallel disks in $V_1$. So $\{\bar{D}_1,\bar{D}_2\}$ cuts $V_1$ into two parts, a $3$-ball and a solid torus. Since $G$ preserves each $M_i$ for $i=1,2$, it also preserves both the ball and the solid torus.
So $G\cong \mathbb{Z}_2$ or $\mathbb{D}_2$ by Lemma \ref{lemma-gm-2}.  This completes the proof of Theorem \ref{main-theorem} when the JSJ-tori and the Heegaard splitting satisfy Condition A in the two non-separating JSJ-tori case.

\section{Proof of Theorem \ref{main-theorem} in the case of Condition B with one separating JSJ-torus\label{section3-1}}
In this section, we will prove Theorem \ref{main-theorem} when the JSJ-tori and the Heegaard splitting satisfy Condition B in the one separating JSJ-torus case.

Let $T$ be the JSJ-torus. By this JSJ-decomposition, we will denote that $M=M_1 \cup_T M_2$  where each $M_i$ is a piece of the JSJ-decomposition for $i=1,2$.

At first, we consider the case when $T\cap V_1$ consists of minimal number of annuli.

If the Heegaard splitting and the JSJ-decomposition satisfy Condition B and the number of annuli are minimal, then each annulus must be induced from each disk of Case 1 or Case 2 in section \ref{section2-1} by an inverse operation of isotopy of type A. Moreover, we have seen that $T\cap V_1$ consists of a disk or two disks in Case 1 or Case 2 of section \ref{section2-1}. So the proof is divided into two cases.

\Case{1} $T\cap V_1$ consists of an annulus $A$. Since $A$ is induced from the disk $\bar{D}$ of Condition A, $\bar{D}$ is a separating disk in $V_1$ as in section \ref{section2-1}. So $A$ cuts $V_1$ into a genus two handlebody $V_1^1$ and a solid torus $V_1^2$ (see Figure \ref{figure-th21-1}.) Since $G$ preserves each piece of the JSJ-decomposition, it also preserves both $V_1^1$ and $V_1^2$. Hence we get $G\cong \mathbb{Z}_2$ or $\mathbb{D}_2$ by Lemma \ref{lemma-gm-an-1} ($V_1^1$ has a meridian disk of $V_1$.)

\begin{figure}
\centering
\includegraphics[viewport=28 516 415 777, width=7cm]{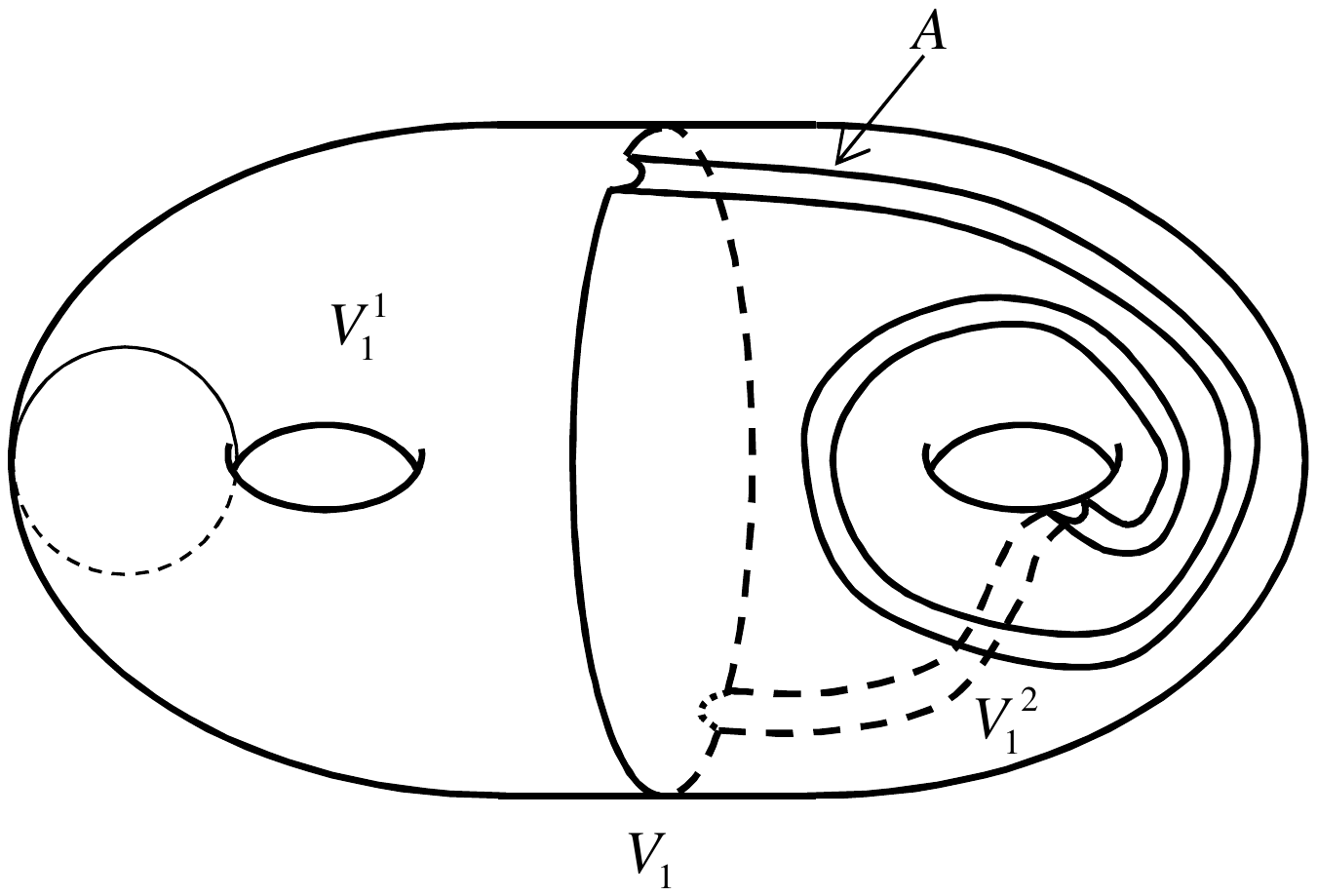}
\caption{$T\cap V_1$ consists of an annulus.\label{figure-th21-1}}
\end{figure}

\Case{2} $T\cap V_1$ consists of two annuli $A_1$ and $A_2$. Since each $A_i$ is induced from the disk $\bar{D}_i$ of Condition A for $i=1,2$, $\{A_1, A_2\}$ satisfies the conclusion (\ref{lemma-ko-3-4-1}) or (\ref{lemma-ko-3-4-2}) of Lemma \ref{lemma-ko-3-4} by the arguments of section \ref{section2-1} (see (i) or (ii) of Figure \ref{fig-lemma-ko-3-4}.) In the former case, $\{A_1,A_2\}$ cuts $V_1$ into a genus two handlebody $V_1^1$ and a solid torus $V_1^2$. Since $G$ preserves each piece of the JSJ-decomposition, it also preserves both $V_1^1$ and $V_1^2$. Hence we get $G\cong \mathbb{Z}_2$ or $\mathbb{D}_2$ by Lemma \ref{lemma-gm-an-1} ($V_1^1$ has a meridian disk of $V_1$.) In the latter case, $\{A_1,A_2\}$ cuts $V_1$ into solid tori $V_1^1$ and $V_1^3$ and a genus two handlebody $V_1^2$, where the gluing of the pieces follows the order of numberings. Therefore, $V_1$ and $V_3$ are contained in a piece of the JSJ-decomposition and $V_2$ is contained in the other piece of the JSJ-decomposition. So $G$ preserves $V_1^i$ for $i=1,2,3$, or $G$ preserves $V_1^1$ and exchanges $V_1^1$ and $V_1^3$. So $G\cong \mathbb{Z}_2$ or $\mathbb{D}_2$ by Lemma \ref{lemma-gm-an-2} (in (ii) of Figure \ref{fig-lemma-ko-3-4}, we can take a meridian disk of $V_1$ such that it intersects only one of $V_1^1$ and $V_1^3$, and also intersects $V'=\overline{V_1}\cup\overline{V_2}$ (or $V'=\overline{V_3}\cup\overline{V_2}$) by a set of meridian disks of $V'$.)

Next, we consider the case when $T\cap V_1$ consists of non-minimal number of annuli.

Since Condition B assumes that $T\cap V_1$ consists of at most two annuli and the annuli are not $\partial$-parallel in $V_1$, we only need to check the case when $T\cap V_1$ consists of two essential annuli $A_1$ and $A_2$ in $V_1$ in general situation. In order to prove, we need the following lemma.

\begin{lemma}\label{lemma-ko-3-2}(\cite{KO}, Lemma 3.2) If $A$ is an essential annulus in a genus two handlebody $V$ then either
    \begin{enumerate}[(i)]
        \item $A$ cuts $V$ into a solid torus $V_1$ and a genus two handlebody $V_2$ and there is a complete system of meridian disks $\{D_1,D_2\}$ of $V_2$ such that $D_1\cap A=\emptyset$ and $D_2\cap A$ is an essential arc of $A$ (see the upper of Figure \ref{fig-lemma-ko-3-2},) or\label{lemma-ko-3-2-1}
        \item $A$ cuts $V$ into a genus two handlebody $V'$ and there is a complete system of meridian disks $\{D_1,D_2\}$ of $V'$ such that $D_1\cap A$ is an essential arc of $A$ (see the lower of Figure \ref{fig-lemma-ko-3-2}.)\label{lemma-ko-3-2-2}
    \end{enumerate}
\end{lemma}

\begin{figure}
\centering
\includegraphics[viewport=31 128 508 780, width=7cm]{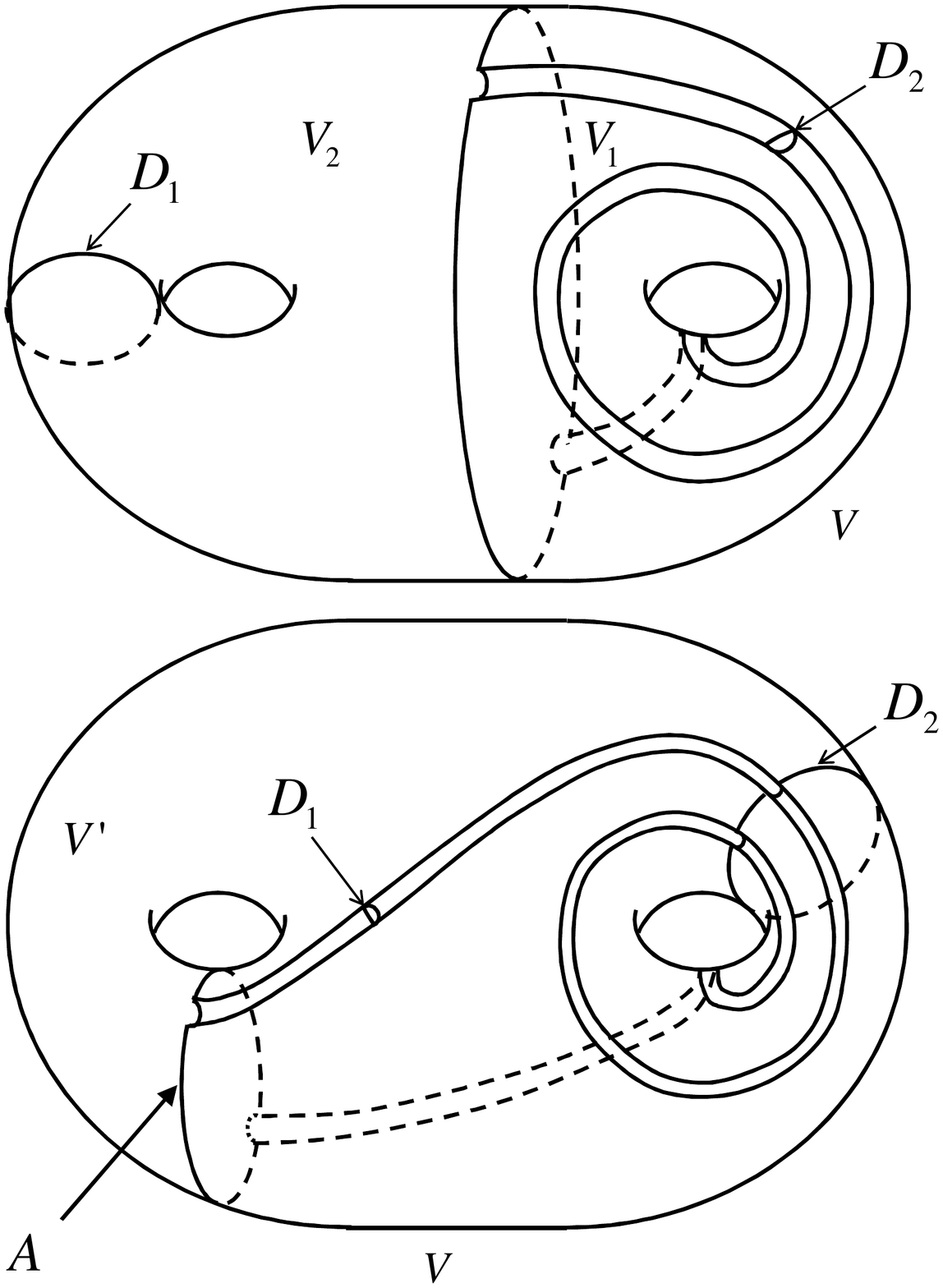}
\caption{Two possible positions of an essential annulus in a genus two handlebody.\label{fig-lemma-ko-3-2}}
\end{figure}

\Case{3} $A_1$ and $A_2$ are parallel in $V_1$. If $A_1$ is non-separating in $V_1$, then $\{A_1,A_2\}$ cuts $V_1$ into a genus two handlebody $V_1^1$ and a solid torus $V_1^2$ (see (ii) of Lemma \ref{lemma-ko-3-2} and (ii) of Figure \ref{figure-add-ad3}.) In this case, $V_1^1$ has a meridian disk of $V_1$ and both $V_1^1$ and $V_1^2$ are non-contractible in $V_1$, so we get $G\cong \mathbb{Z}_2$ or $\mathbb{D}_2$ by Lemma \ref{lemma-gm-an-1} (we can assume $V_1^1$ has a meridian disk of $V_1$ since we can take a meridian disk which is parallel to the disk obtained by a surgery along an essential arc on $A_1$ or $A_2$.)

If $A_1$ is separating in $V_1$, then $\{A_1,A_2\}$ cuts $V_1$ into a genus two handlebody $V_1^1$ and two solid tori $V_1^2$ and $V_1^3$ (see (i) of Lemma \ref{lemma-ko-3-2} and (i) of Figure \ref{figure-add-ad3}.) In this case, the only genus two handlebody component $V_1^1$ is preserved by $G$. Let $A_1$ be $\overline{V_1^1}\cap\overline{V_1^2}$ and $A_2$ be $\overline{V_1^3}\cap\overline{V_1^2}$ as Figure \ref{figure-add-ad3}. Then $\overline{V_1^2}$ meets both $A_1$ and $A_2$, but $\overline{V_1^3}$ meets only $A_2$, so $G$ cannot exchange $V_1^2$ and $V_1^3$ (since $A_1\cup A_2=T\cap V_1$,  $A_1\cup A_2$ is preserved by $G$.) So both $V_1^2$ and $V_1^3$ are preserved by $G$. Let $\bar{V}_1$ be the genus two handlebody $\overline{V_1^1}\cup \overline{V_1^2}$. Then $V_1^1$ has a meridian disk of $\bar{V}_1$ and both $V_1^1$ and $V_1^2$ are non-contractible in $\bar{V}_1$, so we get $G\cong \mathbb{Z}_2$ or $\mathbb{D}_2$ by Lemma \ref{lemma-gm-an-1}.

\begin{figure}
\centering
\includegraphics[viewport=26 231 416 780, width=7cm]{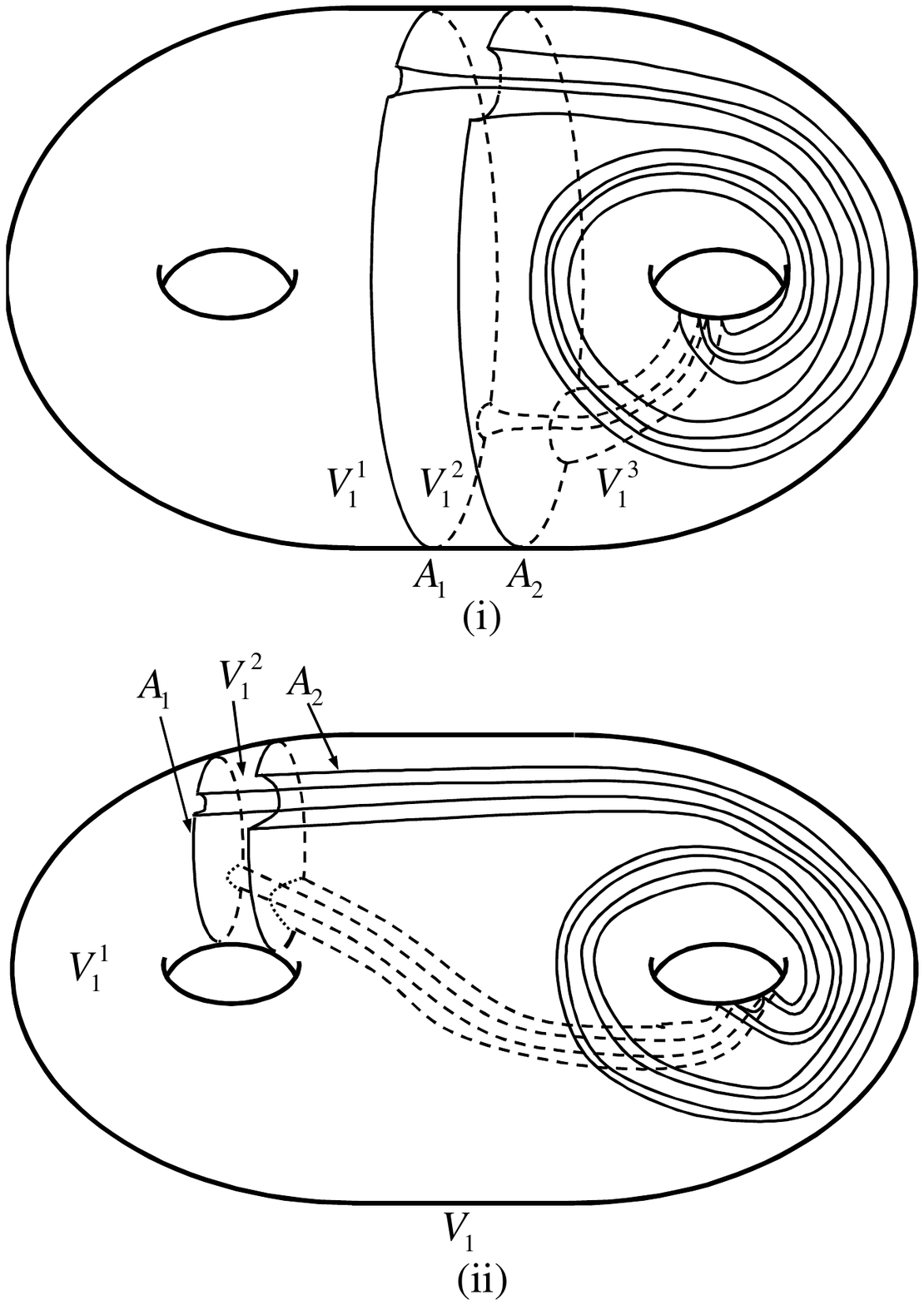}
\caption{$A_1$ and $A_2$ are parallel in $V_1$.\label{figure-add-ad3}}
\end{figure}

\Case{4} $A_1$ and $A_2$ are not parallel in $V_1$. Then $\{A_1,A_2\}$ satisfies one of conclusions of Lemma \ref{lemma-ko-3-4}. But (\ref{lemma-ko-3-4-3}) of Lemma \ref{lemma-ko-3-4} contradicts the fact the JSJ-torus $T$ separates $M$ into two pieces using similar arguments in Case 5 of section \ref{section2-1} (one of $A_1$ and $A_2$ is non-separating in $V_1$, and the other is separating in $V_1$.) So $\{A_1,A_2\}$ satisfies (\ref{lemma-ko-3-4-1}) or (\ref{lemma-ko-3-4-2}) of Lemma \ref{lemma-ko-3-4}. Therefore, we get $G\cong \mathbb{Z}_2$ or $\mathbb{D}_2$ using the arguments in Case 2.

This completes the proof of Theorem \ref{main-theorem} when the JSJ-tori and the Heegaard splitting satisfy Condition B in the one separating JSJ-torus case.

\section{Proof of Theorem \ref{main-theorem} in the case of Condition B with two separating JSJ-tori\label{section3-2}}

In this section, we will prove Theorem \ref{main-theorem} when the JSJ-tori and the Heegaard splitting satisfy Condition B in the two separating JSJ-tori case.

Let $T_1$ and $T_2$ be the JSJ-tori. By this JSJ-decomposition, we will denote that $M=M_1 \cup_{T_1} M_2 \cup_{T_2}M_3$ where each $M_i$ is a piece of the JSJ-decomposition for $i=1,2,3$. Let $A_1$ and $A_2$ be the annuli $T_1\cap V_1$ and $T_2\cap V_1$.

If the Heegaard splitting and the JSJ-decomposition satisfy Condition B, then each annulus $A_i$ must be induced from each disk $\bar{D}_i$ of Condition A in section \ref{section2-2} by an inverse operation of isotopy of type A. Moreover, we have seen that $\{A_1,A_2\}$ satisfies the conclusion (\ref{lemma-ko-3-4-2}) of Lemma \ref{lemma-ko-3-4} in section \ref{section2-2}, i.e. $\{A_1,\,A_2\}$ cuts $V_1$ into two solid tori $V_1^1$ and $V_1^3$, and a genus two handlebody $V_1^2$ (see (ii) of Figure \ref{fig-lemma-ko-3-4}.) Since $G$ preserves each piece of the JSJ-decomposition, it also preserves each $V_1^i$ for $i=1,2,3$. So $G\cong \mathbb{Z}_2$ or $\mathbb{D}_2$ by Lemma \ref{lemma-gm-an-2}. This completes the proof of Theorem \ref{main-theorem} when the JSJ-tori and the Heegaard splitting satisfy Condition B in the two separating JSJ-tori case.

\section{Proof of Theorem \ref{main-theorem} in the case of Condition B with two non-separating JSJ-tori.\label{section3-3}}

In this section, we will prove Theorem \ref{main-theorem} when the JSJ-tori and the Heegaard splitting satisfy Condition B in the two non-separating JSJ-tori case.

Let $T_1$ and $T_2$ be the JSJ-tori. By this JSJ-decomposition, we will denote that $M=M_1 \cup_{T_1, T_2} M_2$ where each $M_i$ is a piece of the JSJ-decomposition for $i=1,2$. Let $A_1$ and $A_2$ be the annuli $T_1\cap V_1$ and $T_2\cap V_1$.

If the Heegaard splitting and the JSJ-decomposition satisfy Condition B, then each annulus $A_i$ must be induced from each disk $\bar{D}_i$ of Condition A in section \ref{section2-3} by an inverse operation of isotopy of type A. We have seen that $\bar{D}_1$ and $\bar{D}_2$ are non-separating, disjoint and  parallel disks in $V_1$ from Lemma \ref{lemma-thm-3}. Since each $\bar{D}_i$ is non-separating in $V_1$ for $i=1,2$, each $A_i$ is also non-separating in $V_1$ for $i=1,2$ (see Figure \ref{fig-lemma-ko-3-2}.) Now the proof is divided into two cases.

\Case{1} $A_1$ and $A_2$ are parallel in $V_1$.

Since each $A_i$ is non-separating in $V_1$ for $i=1,2$, we get the conclusion (\ref{lemma-ko-3-2-2}) of Lemma \ref{lemma-ko-3-2} for $A_1$ (so for $A_2$). So $\{A_1,\,A_2\}$ cuts $V_1$ into a genus two handlebody $V_1^1$ and a solid torus $V_1^2$ (see (ii) of Figure \ref{figure-add-ad3}.) Since $G$ preserves each piece of the JSJ-decomposition, it also preserves both $V_1^1$ and $V_1^2$. So we get $G\cong \mathbb{Z}_2$ or $\mathbb{D}_2$ by Lemma \ref{lemma-gm-an-1}.

\Case{2} $A_1$ and $A_2$ are non-parallel in $V_1$. Then $\{A_1, A_2\}$ satisfies the conclusion (\ref{lemma-ko-3-4-1}) of Lemma \ref{lemma-ko-3-4} since $\bar{D}_1$ and $\bar{D}_2$ are both non-separating in $V_1$, i.e. $\{A_1,\,A_2\}$ cuts $V_1$ into a genus two handlebody $V_1^1$ and a solid torus $V_1^2$ (see (i) of Figure \ref{fig-lemma-ko-3-4}.) Since $G$ preserves each piece of the JSJ-decomposition, it also preserves both $V_1^1$ and $V_1^2$. So we get $G\cong \mathbb{Z}_2$ or $\mathbb{D}_2$ by Lemma \ref{lemma-gm-an-1}.

This completes the proof of Theorem \ref{main-theorem} when the JSJ-tori and the Heegaard splitting satisfy Condition B in the two non-separating JSJ-tori case.

\bibliographystyle{plain}

\begin{thebibliography}{10}

\bibitem{JA} W. Jaco, \emph{Lectures on three manifold topology},
Conference Board of Math. Science, Regional Conference Series in
Math. No. \textbf{43}, 1980.

\bibitem{KM} J. Kalliongis and A. Miller, \emph{The symmetries of genus one handlebodies}, Can. J.
Math., \textbf{43} (1991), no. 2, 371-404.

\bibitem{KJS} Jungsoo Kim, \textit{Structures of geometric quotient orbifolds of three-dimensional $G$-manifolds with Heegaard genus two}, It will appear on J. Korean Math. Soc. \textbf{46} (2009), No. 4, pp. 859--893, now available on\\ \verb|http://www.kms.or.kr/home/kor/article/journal/journal_015.asp?globalmenu=9&localmenu=7|


\bibitem{KO2} T. Kobayashi, \emph{Non-separating incompressible tori in $3$-manifolds}, J. Math. Soc. Japan, \textbf{36} (1984), no. 1, 11--22.

\bibitem{KO} T. Kobayashi, \emph{Structures of the Haken manifolds with Heegaard splittings of genus two},  Osaka J. Math.  \textbf{21} (1984), 437--455.

\bibitem{KO3} T. Kobayashi, \emph{Casson-Gordon's rectangle condition of Heegaard diagrams and incompressible tori in $3$-manifolds}, Osaka J. Math. \textbf{25} (1988), 553-573.

\bibitem{MMZ} D. McCullough, A. Miller and  B. Zimmermann, \emph{Group actions on handlebodies}, Proc. London Math. Soc. \textbf{59} (1989), 373--416.

\bibitem{MY} W. H. Meeks and S.-T. Yau, The equivariant loop theorem for three-dimensional manifolds, The Smith conjecture (J. Morgan, H. Bass, eds.), Academic Press, Orlando, Fl, 1984.

\bibitem{Z1} B. Zimmermann, \emph{Genus actions of finite groups on $3$-manifolds}, Michigan Math. J. \textbf{43} (1996) 593--610.

\end{thebibliography}

\end{document}